\documentclass[oneside,english]{amsart}
\usepackage[T1]{fontenc}
\usepackage[latin9]{inputenc}
\usepackage[active]{srcltx}
\usepackage{float}
\usepackage{amsthm}
\usepackage{setspace}
\makeatletter

\providecommand{\tabularnewline}{\\}

\numberwithin{equation}{section}
\numberwithin{figure}{section}
\theoremstyle{plain}
\newtheorem{thm}{\protect\theoremname}
  \theoremstyle{plain}
  \newtheorem{lem}[thm]{\protect\lemmaname}
  \theoremstyle{plain}
  \newtheorem{cor}[thm]{\protect\corollaryname}

\usepackage{tikz}

\makeatother

\usepackage{babel}
  \providecommand{\corollaryname}{Corollary}
  \providecommand{\lemmaname}{Lemma}
\providecommand{\theoremname}{Theorem}

\begin{document}

\title{Eternal Picaria}

\author{Urban Larsson}

\author{Israel Rocha}

\thanks{Urban Larsson, urban031@gmail.com (partly supported by the Killam
Trust) and Israel Rocha, israel.rocha@dal.ca, Department of Mathematics
and Statistics, Dalhousie University, Halifax, Canada.}
\begin{abstract}
Picaria is a traditional board game, played by the Zuni tribe of the
American Southwest and other parts of the world, such as a rural Southwest
region in Sweden. It is related to the popular children's game of
Tic-tac-toe, but the 2 players have only 3 stones each, and in the
second phase of the game, pieces are slided, along specified move
edges, in attempts to create the three-in-a-row. We provide a rigorous
solution, and prove that the game is a draw; moreover our solution
gives insights to strategies that players can use.
\end{abstract}

\maketitle

\section{Introduction}

Picaria is a traditional board game, played by the Zuni tribe of the
American Southwest and other parts of the world, such as in a rural
Southwest region in Sweden\footnote{The `first' author played this game as a child with his grandparents
in the village Rångedala close to the Swedish city Borås, and the
game was called ``luffarschack''.}. The game is related to the popular children's game of Tic-tac-toe,
but it is even more related to other three-in-a-row games such as
Three men's morris, Tapatan, Nine Holes, Achi, Tant Fant and Shisma.
These latter games are sometimes played in two phases, the first phase
being placement of stones, and the second part being sliding of stones
along prescribed `move edges'. In either case, the possibility of
infinite play puts Picaria in a different class than Tic-tac-toe. 

The `blockade' games of Pong Hau K'i, from China, and Mu Torere, played
by the M\={a}ori people from the east coast of New Zealand's North
Island, are also related (in these games, if a player cannot move,
he loses); with quite few positions, only 16 and 46 respectively,
these games are solved \cite{key-1} by depicting the position graphs.

In Picaria, there are two players who alternate turns, and the goal
is to be the first player to place 3 game pieces of a kind in a row,
vertically, horizontally or diagonally. Each player has their own
type of pieces, say, to use the convention of Tic-tac-toe, X and O.
In our study, we assume that player X starts. The players alternate
turn to (in phase 1) place their stones in an open space in a 3 by
3 grid. When each piece has been played in the first phase (and assuming
a non-loss so far), then player X begins the second phase by sliding
one of the three Xs to an empty adjacent node; then O slides a stone,
and so on. Adjacency here means a neighboring node, horizontally,
vertically or diagonally. A game position is declared a draw if periodicity
of a pattern is forced by one of the players. In this paper we give
a constructive proof that the opening position (the empty board) is
a draw\footnote{The historical popularity of the game is probably due to the fact
that it is very easy to make a single mistake and then there is a
human player winner. As a children's game, by the author's experience,
the outcome is rarely a draw. After many plays though, over several
years, there is a kind of certainty that both players should be able
to draw, and here we show how. By this play heuristics we believe
that any proof would be non-trivial.}.

Picaria was described in the literature for the first time in 1907
by the ethnographer Stewart Culin \cite{key-2}. The original board
of Picaria is displayed in Figure \ref{fig:originalBoard}. The players
place stones on the vertices of the graph and slide along the edges.
In this paper we play the game in an equivalent manner using instead
a Tic-tac-toe board.\bigskip

\begin{center}
\begin{figure}[H]
\begin{centering}
\begin{tikzpicture}[scale=0.10,baseline=10] 
\draw[fill] (0,0) circle [radius=.6]; 
\draw[fill] (10,0) circle [radius=.6]; 
\draw[fill] (20,0) circle [radius=.6]; 
\draw[fill] (0,10) circle [radius=.6]; 
\draw[fill] (10,10) circle [radius=.6]; 
\draw[fill] (20,10) circle [radius=.6]; 
\draw[fill] (0,20) circle [radius=.6]; 
\draw[fill] (10,20) circle [radius=.6]; 
\draw[fill] (20,20) circle [radius=.6]; 
\draw (0,0) -- (20,0); 
\draw (0,10) -- (20,10); 
\draw (0,20) -- (20,20); 
\draw (0,0) -- (0,20); 
\draw (10,0) -- (10,20); 
\draw (20,0) -- (20,20); 
\draw (0,0) -- (20,20); 
\draw (0,20) -- (20,0); 
\draw (0,10) -- (10,0); 
\draw (10,20) -- (20,10); 
\draw (0,10) -- (10,20); 
\draw (10,0) -- (20,10); 

\draw (40,6.5) -- (60,6.5); 
\draw (40,13.5) -- (60,13.5); 
\draw (46.5,0) -- (46.5,20); 
\draw (53.5,0) -- (53.5,20); 

\end{tikzpicture}\vspace{8 mm}
\par\end{centering}

\caption{The original Picaria board and a Tic-tac-toe board}
\label{fig:originalBoard}
\end{figure}
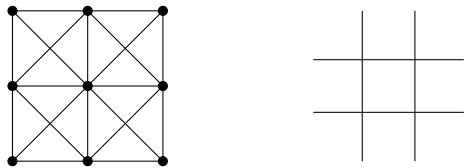

\par\end{center}

\vspace{-1.5 cm}Before all stones are on the board, the number of
positions coincide with those of Tic-tac-toe. Subsequently there are
456 positions modulo symmetries (see Section \ref{sec:counting} for
details) and many of those positions could be revisited during the
course of a game, so notably play is very different from Tic-tac-toe.
Even though a computer could be used to solve Picaria, this would
not provide a full understanding on how to play a succesful strategy.
In this paper we give the explicit \textit{strategies} of optimal
play, which means that perfect information players will not play to
draw if they can win and they will not play to lose if they can avoid
loss. The latter idea will be useful at particular stages of our play
proofs. In optimal play, if a player revisits a position, then the
game is a draw. It turns out that both players are able to draw from
the initial position. As a consequence, efficient strategies to win
by for example using \textit{Fork}-, \textit{Trap-, Race-, or Zugzwang-positions}
(as illustrated), will \textsl{not} occur in optimal play

\noindent \begin{center}
$\xrightarrow{x}$\textsf{}%
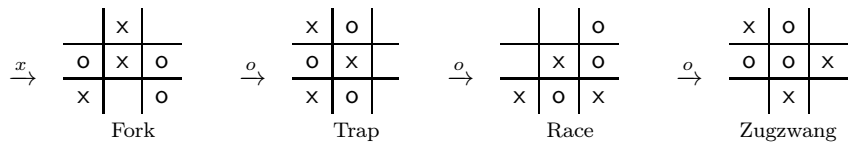
\begin{tabular}{cc|c|cc}
 &  & \textsf{x} &  & \tabularnewline
\cline{2-4} 
 & \textsf{o} & \textsf{x} & \textsf{o} & \;\;\;\tabularnewline
\cline{2-4} 
 & \textsf{x} &  & \textsf{o} & \tabularnewline
\multicolumn{5}{c}{{\footnotesize{}Fork}}\tabularnewline
\end{tabular} \textsf{$\xrightarrow{o}$}%
\begin{tabular}{cc|c|cc}
 & \textsf{x} & \textsf{o} &  & \tabularnewline
\cline{2-4} 
\!\! & \textsf{o} & \textsf{x} &  & \;\;\;\tabularnewline
\cline{2-4} 
 & \textsf{x} & \textsf{o} &  & \tabularnewline
\multicolumn{5}{c}{{\footnotesize{}Trap}}\tabularnewline
\end{tabular}\textsf{$\xrightarrow{o}$}%
\begin{tabular}{cc|c|cc}
 &  &  & \textsf{o} & \tabularnewline
\cline{2-4} 
\!\! &  & \textsf{x} & \textsf{o} & \;\;\;\tabularnewline
\cline{2-4} 
 & \textsf{x} & \textsf{o} & \textsf{x} & \tabularnewline
\multicolumn{5}{c}{{\footnotesize{}Race}}\tabularnewline
\end{tabular} \textsf{$\xrightarrow{o}$}%
\begin{tabular}{cc|c|cc}
 & \textsf{x} & \textsf{o} &  & \tabularnewline
\cline{2-4} 
\!\! & \textsf{o} & \textsf{o} & \textsf{x} & \tabularnewline
\cline{2-4} 
 &  & \textsf{x} &  & \tabularnewline
\multicolumn{5}{c}{{\footnotesize{}Zugzwang}}\tabularnewline
\end{tabular}
\par\end{center}

Note here that a game position is depicted as a game board together
with a flag for who just moved. Often however we omit the move-flag,
because it is clear by the context which play position we discuss
(for example in the placement phase of the game). For example, play
just before the Trap-position involves a bad move by X 

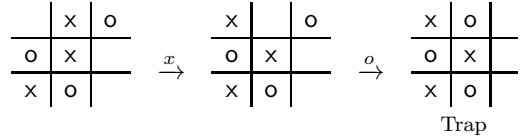
\begin{center}
\textsf{}%
\begin{tabular}{cc|c|cc}
 &  & \textsf{x} & \textsf{o} & \tabularnewline
\cline{2-4} 
 & \textsf{o} & \textsf{x} &  & \tabularnewline
\cline{2-4} 
 & \textsf{x} & \textsf{o} &  & \tabularnewline
\multicolumn{5}{c}{}\tabularnewline
\end{tabular}\textsf{$\xrightarrow{x}$}%
\begin{tabular}{cc|c|cc}
 & \textsf{x} &  & \textsf{o} & \tabularnewline
\cline{2-4} 
 & \textsf{o} & \textsf{x} &  & \tabularnewline
\cline{2-4} 
 & \textsf{x} & \textsf{o} &  & \tabularnewline
\multicolumn{5}{c}{}\tabularnewline
\end{tabular}\textsf{$\xrightarrow{o}$}%
\begin{tabular}{cc|c|cc}
 & \textsf{x} & \textsf{o} &  & \tabularnewline
\cline{2-4} 
 & \textsf{o} & \textsf{x} &  & \tabularnewline
\cline{2-4} 
 & \textsf{x} & \textsf{o} &  & \tabularnewline
\multicolumn{5}{c}{{\footnotesize{}Trap}}\tabularnewline
\end{tabular}
\par\end{center}

Such play does not belong to X's strategy, and similar ideas are commonly
applied inside proofs.

\section{A Rigorous Play-analysis of Picaria}

In our convention X is the first player. We show that player O can
prevent X from winning, by forcing X to play a periodic sequence of
moves. That would prove that the game is a draw if X can prevent O
from winning too. We begin by proving that it is easy for X to avoid
loss.

\subsection{The second player cannot win Picaria}
\begin{thm}
\label{lem:A}Player $O$ cannot win.\end{thm}
\begin{proof}
Player X starts by playing in the center and then there are two cases
to consider \bigskip

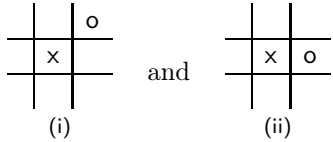
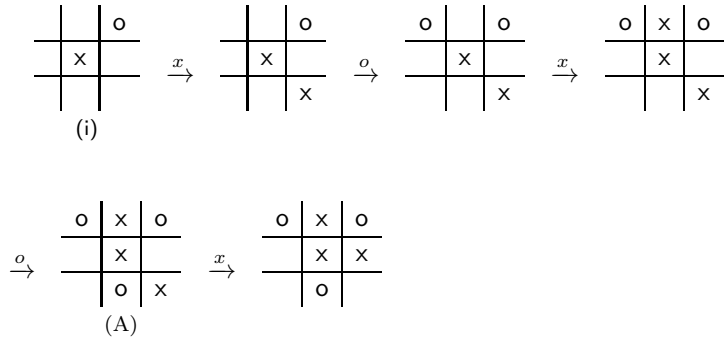
\textsf{}%
\begin{tabular}{cc|c|cc}
 &  &  & \textsf{o} & \tabularnewline
\cline{2-4} 
 &  & \textsf{x} &  & \tabularnewline
\cline{2-4} 
 &  &  &  & \tabularnewline
\multicolumn{5}{c}{\textsf{\footnotesize{}(i)}}\tabularnewline
\end{tabular} and \textsf{}%
\begin{tabular}{cc|c|cr@{\extracolsep{0pt}.}l}
 &  &  &  & \multicolumn{2}{c}{}\tabularnewline
\cline{2-4} 
 &  & \textsf{x} & \textsf{o} & \multicolumn{2}{c}{}\tabularnewline
\cline{2-4} 
 &  &  &  & \multicolumn{2}{c}{}\tabularnewline
\multicolumn{6}{c}{\textsf{\footnotesize{}(ii)}}\tabularnewline
\end{tabular}\textsf{}\\

For game (i)

\bigskip

\textsf{}%
\begin{tabular}{cc|c|cc}
 &  &  & \textsf{o} & \tabularnewline
\cline{2-4} 
 &  & \textsf{x} &  & \tabularnewline
\cline{2-4} 
 &  &  &  & \tabularnewline
\multicolumn{5}{c}{\textsf{\footnotesize{}(i)}}\tabularnewline
\end{tabular}\textsf{$\xrightarrow{x}$}%
\begin{tabular}{cc|c|cc}
 &  &  & \textsf{o} & \tabularnewline
\cline{2-4} 
 &  & \textsf{x} &  & \tabularnewline
\cline{2-4} 
 &  &  & \textsf{x} & \tabularnewline
\multicolumn{5}{c}{}\tabularnewline
\end{tabular}\textsf{$\xrightarrow{o}$}%
\begin{tabular}{cc|c|cc}
 & \textsf{o} &  & \textsf{o} & \tabularnewline
\cline{2-4} 
 &  & \textsf{x} &  & \tabularnewline
\cline{2-4} 
 &  &  & \textsf{x} & \tabularnewline
\multicolumn{5}{c}{}\tabularnewline
\end{tabular}\textsf{$\xrightarrow{x}$}%
\begin{tabular}{cc|c|cc}
 & \textsf{o} & \textsf{x} & \textsf{o} & \tabularnewline
\cline{2-4} 
 &  & \textsf{x} &  & \tabularnewline
\cline{2-4} 
 &  &  & \textsf{x} & \tabularnewline
\multicolumn{5}{c}{}\tabularnewline
\end{tabular}\textsf{}\\

\bigskip\textsf{$\xrightarrow{o}$}%
\begin{tabular}{cc|c|cc}
 & \textsf{o} & \textsf{x} & \textsf{o} & \tabularnewline
\cline{2-4} 
 &  & \textsf{x} &  & \tabularnewline
\cline{2-4} 
 &  & \textsf{o} & \textsf{x} & \tabularnewline
\multicolumn{5}{c}{{\footnotesize{}(A)}}\tabularnewline
\end{tabular}\textsf{$\xrightarrow{x}$}%
\begin{tabular}{cc|c|cc}
 & \textsf{o} & \textsf{x} & \textsf{o} & \tabularnewline
\cline{2-4} 
 &  & \textsf{x} & \textsf{x} & \tabularnewline
\cline{2-4} 
 &  & \textsf{o} &  & \tabularnewline
\multicolumn{5}{c}{}\tabularnewline
\end{tabular} 

\bigskip Thus, X can force a return to the game (A), depicted above,
by 

\bigskip

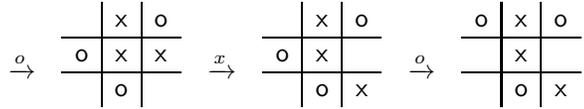
\textsf{$\xrightarrow{o}$}%
\begin{tabular}{cc|c|cc}
 &  & \textsf{x} & \textsf{o} & \tabularnewline
\cline{2-4} 
 & \textsf{o} & \textsf{x} & \textsf{x} & \tabularnewline
\cline{2-4} 
 &  & \textsf{o} &  & \tabularnewline
\multicolumn{5}{c}{}\tabularnewline
\end{tabular}\textsf{$\xrightarrow{x}$}%
\begin{tabular}{cc|c|cc}
 &  & \textsf{x} & \textsf{o} & \tabularnewline
\cline{2-4} 
 & \textsf{o} & \textsf{x} &  & \tabularnewline
\cline{2-4} 
 &  & \textsf{o} & \textsf{x} & \tabularnewline
\multicolumn{5}{c}{}\tabularnewline
\end{tabular}\textsf{$\xrightarrow{o}$}%
\begin{tabular}{cc|c|cc}
 & \textsf{o} & \textsf{x} & \textsf{o} & \tabularnewline
\cline{2-4} 
 &  & \textsf{x} &  & \tabularnewline
\cline{2-4} 
 &  & \textsf{o} & \textsf{x} & \tabularnewline
\multicolumn{5}{c}{}\tabularnewline
\end{tabular} 

\bigskip

On the other hand, game\textsf{ (ii)}%
\begin{tabular}{cc|c|cc}
 &  &  &  & \tabularnewline
\cline{2-4} 
 &  & \textsf{x} & \textsf{o} & \tabularnewline
\cline{2-4} 
 &  &  &  & \tabularnewline
\multicolumn{5}{c}{}\tabularnewline
\end{tabular} is losing for player O, since 

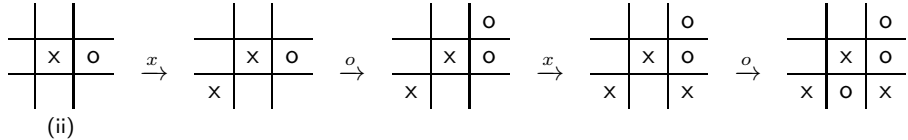
\textsf{}%
\begin{tabular}{cc|c|cc}
 &  &  &  & \tabularnewline
\cline{2-4} 
 &  & \textsf{x} & \textsf{o} & \tabularnewline
\cline{2-4} 
 &  &  &  & \tabularnewline
\multicolumn{5}{c}{\textsf{\footnotesize{}(ii)}}\tabularnewline
\end{tabular}\textsf{$\xrightarrow{x}$}%
\begin{tabular}{cc|c|cc}
 &  &  &  & \tabularnewline
\cline{2-4} 
 &  & \textsf{x} & \textsf{o} & \tabularnewline
\cline{2-4} 
 & \textsf{x} &  &  & \tabularnewline
\multicolumn{5}{c}{}\tabularnewline
\end{tabular}\textsf{$\xrightarrow{o}$}%
\begin{tabular}{cc|c|cc}
 &  &  & \textsf{o} & \tabularnewline
\cline{2-4} 
 &  & \textsf{x} & \textsf{o} & \tabularnewline
\cline{2-4} 
 & \textsf{x} &  &  & \tabularnewline
\multicolumn{5}{c}{}\tabularnewline
\end{tabular}\textsf{$\xrightarrow{x}$}%
\begin{tabular}{cc|c|cc}
 &  &  & \textsf{o} & \tabularnewline
\cline{2-4} 
 &  & \textsf{x} & \textsf{o} & \tabularnewline
\cline{2-4} 
 & \textsf{x} &  & \textsf{x} & \tabularnewline
\multicolumn{5}{c}{}\tabularnewline
\end{tabular}\textsf{$\xrightarrow{o}$}%
\begin{tabular}{cc|c|cc}
 &  &  & \textsf{o} & \tabularnewline
\cline{2-4} 
 &  & \textsf{x} & \textsf{o} & \tabularnewline
\cline{2-4} 
 & \textsf{x} & \textsf{o} & \textsf{x} & \tabularnewline
\multicolumn{5}{c}{}\tabularnewline
\end{tabular} 

\bigskip and this is a Race-position, from which player X wins in
two moves.
\end{proof}

\subsection{The first player cannot win playing from a Loop position}

Next we analyse a special configuration which is quite recurring in
the game, called a \textit{Loop position}:\bigskip

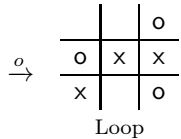
\begin{center}
\textsf{$\xrightarrow{o}$}%
\begin{tabular}{cc|c|cc}
 &  &  & \textsf{o} & \tabularnewline
\cline{2-4} 
 & \textsf{o} & \textsf{x} & \textsf{x} & \tabularnewline
\cline{2-4} 
 & \textsf{x} &  & \textsf{o} & \tabularnewline
\multicolumn{5}{c}{{\footnotesize{}Loop}}\tabularnewline
\end{tabular}\bigskip
\par\end{center}

Any other symmetric position (a rotation or reflection; see for example
game (A) in the proof of Theorem \ref{lem:A}) is also called a Loop
position\footnote{We use loop in the sense of something reappearing, although in graph
theoretical terms it would be more correct to call this a cycle.}. The reason for this will be clear in the next result, Theorem \ref{thm:XLoop},
where we show that O prevents X from winning the game by means of
a periodic sequence of moves. Consider that player X \textit{holds
the center}. This restricts the possibilities for X in that only the
two outer stones can be moves. In particular if X starts from the
Loop position, then player X cannot win, which constitutes our next
lemma. When Picaria is played by human (non-optimal) players, a player
holding the center often appears to enjoy a certain advantage. The
next results also revolve around the idea of a player either holding
or leaving the center.\smallskip
\begin{lem}
\label{lem:Xholdsthecenter}If player X is to move and it refuses
to leave the center starting from a Loop position, then player O can
force a return to this position.\end{lem}
\begin{proof}
There are only three possible moves from a Loop-position for X, since
it holds the center.\bigskip

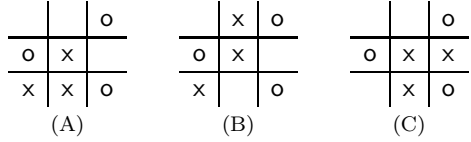
\textsf{}%
\begin{tabular}{cc|c|cc}
 &  &  & \textsf{o} & \tabularnewline
\cline{2-4} 
 & \textsf{o} & \textsf{x} &  & \tabularnewline
\cline{2-4} 
 & \textsf{x} & \textsf{x} & \textsf{o} & \tabularnewline
\multicolumn{5}{c}{{\footnotesize{}(A)}}\tabularnewline
\end{tabular}\textsf{}%
\begin{tabular}{cc|c|cc}
 &  & \textsf{x} & \textsf{o} & \tabularnewline
\cline{2-4} 
 & \textsf{o} & \textsf{x} &  & \tabularnewline
\cline{2-4} 
 & \textsf{x} &  & \textsf{o} & \tabularnewline
\multicolumn{5}{c}{{\footnotesize{}(B)}}\tabularnewline
\end{tabular}\textsf{}%
\begin{tabular}{cc|c|cc}
 &  &  & \textsf{o} & \tabularnewline
\cline{2-4} 
 & \textsf{o} & \textsf{x} & \textsf{x} & \tabularnewline
\cline{2-4} 
 &  & \textsf{x} & \textsf{o} & \tabularnewline
\multicolumn{5}{c}{{\footnotesize{}(C)}}\tabularnewline
\end{tabular}\bigskip

For game (A), player X gets trapped by \bigskip

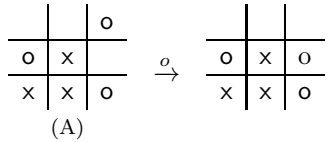
\textsf{}%
\begin{tabular}{cc|c|cc}
 &  &  & \textsf{o} & \tabularnewline
\cline{2-4} 
 & \textsf{o} & \textsf{x} &  & \tabularnewline
\cline{2-4} 
 & \textsf{x} & \textsf{x} & \textsf{o} & \tabularnewline
\multicolumn{5}{c}{{\footnotesize{}(A)}}\tabularnewline
\end{tabular}\textsf{$\xrightarrow{o}$}%
\begin{tabular}{cc|c|cc}
 &  &  &  & \tabularnewline
\cline{2-4} 
 & \textsf{o} & \textsf{x} & o & \tabularnewline
\cline{2-4} 
 & \textsf{x} & \textsf{x} & \textsf{o} & \tabularnewline
\multicolumn{5}{c}{}\tabularnewline
\end{tabular} \bigskip 

So X would have no option but give up the center and clearly loses
the game. Therefore, the game (A) does not belong to X's strategy.

Now, player O can force either of the following sequences \bigskip

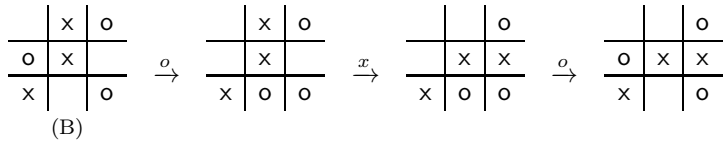
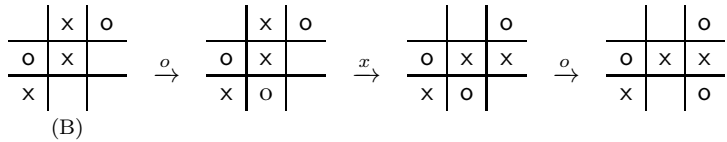
\textsf{}%
\begin{tabular}{cc|c|cc}
 &  & \textsf{x} & \textsf{o} & \tabularnewline
\cline{2-4} 
 & \textsf{o} & \textsf{x} &  & \tabularnewline
\cline{2-4} 
 & \textsf{x} &  & \textsf{o} & \tabularnewline
\multicolumn{5}{c}{{\footnotesize{}(B)}}\tabularnewline
\end{tabular}\textsf{$\xrightarrow{o}$}%
\begin{tabular}{cc|c|cc}
 &  & \textsf{x} & \textsf{o} & \tabularnewline
\cline{2-4} 
 &  & \textsf{x} &  & \tabularnewline
\cline{2-4} 
 & \textsf{x} & \textsf{o} & \textsf{o} & \tabularnewline
\multicolumn{5}{c}{}\tabularnewline
\end{tabular}\textsf{$\xrightarrow{x}$}%
\begin{tabular}{cc|c|cc}
 &  &  & \textsf{o} & \tabularnewline
\cline{2-4} 
 &  & \textsf{x} & \textsf{x} & \tabularnewline
\cline{2-4} 
 & \textsf{x} & \textsf{o} & \textsf{o} & \tabularnewline
\multicolumn{5}{c}{}\tabularnewline
\end{tabular}\textsf{$\xrightarrow{o}$}%
\begin{tabular}{cc|c|cc}
 &  &  & \textsf{o} & \tabularnewline
\cline{2-4} 
 & \textsf{o} & \textsf{x} & \textsf{x} & \tabularnewline
\cline{2-4} 
 & \textsf{x} &  & \textsf{o} & \tabularnewline
\multicolumn{5}{c}{}\tabularnewline
\end{tabular}\bigskip

or \bigskip

\textsf{}%
\begin{tabular}{cc|c|cc}
 &  & \textsf{x} & \textsf{o} & \tabularnewline
\cline{2-4} 
 & \textsf{o} & \textsf{x} &  & \tabularnewline
\cline{2-4} 
 & \textsf{x} &  &  & \tabularnewline
\multicolumn{5}{c}{{\footnotesize{}(B)}}\tabularnewline
\end{tabular}\textsf{$\xrightarrow{o}$}%
\begin{tabular}{cc|c|cc}
 &  & \textsf{x} & \textsf{o} & \tabularnewline
\cline{2-4} 
 & \textsf{o} & \textsf{x} &  & \tabularnewline
\cline{2-4} 
 & \textsf{x} & o &  & \tabularnewline
\multicolumn{5}{c}{}\tabularnewline
\end{tabular}\textsf{$\xrightarrow{x}$}%
\begin{tabular}{cc|c|cc}
 &  &  & \textsf{o} & \tabularnewline
\cline{2-4} 
 & \textsf{o} & \textsf{x} & \textsf{x} & \tabularnewline
\cline{2-4} 
 & \textsf{x} & \textsf{o} &  & \tabularnewline
\multicolumn{5}{c}{}\tabularnewline
\end{tabular}\textsf{$\xrightarrow{o}$}%
\begin{tabular}{cc|c|cc}
 &  &  & \textsf{o} & \tabularnewline
\cline{2-4} 
 & \textsf{o} & \textsf{x} & \textsf{x} & \tabularnewline
\cline{2-4} 
 & \textsf{x} &  & \textsf{o} & \tabularnewline
\multicolumn{5}{c}{}\tabularnewline
\end{tabular}\bigskip

which is the initial (Loop) position. Note that in the second diagram
X does not move into the Trap position. Finally, \bigskip

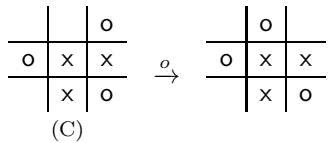
\textsf{}%
\begin{tabular}{cc|c|cc}
 &  &  & \textsf{o} & \tabularnewline
\cline{2-4} 
 & \textsf{o} & \textsf{x} & \textsf{x} & \tabularnewline
\cline{2-4} 
 &  & \textsf{x} & \textsf{o} & \tabularnewline
\multicolumn{5}{c}{{\footnotesize{}(C)}}\tabularnewline
\end{tabular}\textsf{$\xrightarrow{o}$}%
\begin{tabular}{cc|c|cc}
 &  & \textsf{o} &  & \tabularnewline
\cline{2-4} 
 & \textsf{o} & \textsf{x} & \textsf{x} & \tabularnewline
\cline{2-4} 
 &  & \textsf{x} & \textsf{o} & \tabularnewline
\multicolumn{5}{c}{}\tabularnewline
\end{tabular} \bigskip

Now by symmetry X moves to \bigskip

\textsf{$\xrightarrow{x}$}%
\begin{tabular}{cc|c|cc}
 &  & \textsf{o} & \textsf{x} & \tabularnewline
\cline{2-4} 
 & \textsf{o} & \textsf{x} &  & \tabularnewline
\cline{2-4} 
 &  & \textsf{x} & \textsf{o} & \tabularnewline
\multicolumn{5}{c}{}\tabularnewline
\end{tabular}\textsf{$\xrightarrow{o}$}%
\begin{tabular}{cc|c|cc}
 &  & \textsf{o} & \textsf{x} & \tabularnewline
\cline{2-4} 
 &  & \textsf{x} &  & \tabularnewline
\cline{2-4} 
 & \textsf{o} & \textsf{x} & \textsf{o} & \tabularnewline
\multicolumn{5}{c}{}\tabularnewline
\end{tabular} \bigskip

which returns to the initial Loop-position.
\end{proof}
The next lemma concerns a `dual' result for a Loop-position.\bigskip
\begin{lem}
\label{lem:LemmaOloop}If O is to move from \textsf{}%
\begin{tabular}{cc|c|cc}
 &  &  & \textsf{x} & \tabularnewline
\cline{2-4} 
 & \textsf{x} & \textsf{o} & \textsf{o} & \tabularnewline
\cline{2-4} 
 & \textsf{o} &  & \textsf{x} & \tabularnewline
\multicolumn{5}{c}{}\tabularnewline
\end{tabular} then O can force a Loop.\end{lem}
\begin{proof}
Player O begins with\bigskip

\textsf{}%
\begin{tabular}{cc|c|cc}
 &  &  & \textsf{x} & \tabularnewline
\cline{2-4} 
 & \textsf{x} & \textsf{o} & \textsf{o} & \tabularnewline
\cline{2-4} 
 & \textsf{o} &  & \textsf{x} & \tabularnewline
\multicolumn{5}{c}{}\tabularnewline
\end{tabular}\textsf{$\xrightarrow{o}$}%
\begin{tabular}{cc|c|cc}
 &  & \textsf{o} & \textsf{x} & \tabularnewline
\cline{2-4} 
 & \textsf{x} & \textsf{o} &  & \tabularnewline
\cline{2-4} 
 & \textsf{o} &  & \textsf{x} & \tabularnewline
\multicolumn{5}{c}{}\tabularnewline
\end{tabular} \bigskip

Now X has two defense possibilities\bigskip

\textsf{$\xrightarrow{x}$}%
\begin{tabular}{cc|c|cc}
 &  & \textsf{o} & \textsf{x} & \tabularnewline
\cline{2-4} 
 &  & \textsf{o} &  & \tabularnewline
\cline{2-4} 
 & \textsf{o} & \textsf{x} & \textsf{x} & \tabularnewline
\multicolumn{5}{c}{{\footnotesize{}(A)}}\tabularnewline
\end{tabular}or \textsf{$\;\;\xrightarrow{x}$}%
\begin{tabular}{cc|c|cc}
 &  & \textsf{o} & \textsf{x} & \tabularnewline
\cline{2-4} 
 & \textsf{x} & \textsf{o} &  & \tabularnewline
\cline{2-4} 
 & \textsf{o} & \textsf{x} &  & \tabularnewline
\multicolumn{5}{c}{{\footnotesize{}(B)}}\tabularnewline
\end{tabular}\bigskip

For game (A) \bigskip

\textsf{}%
\begin{tabular}{cc|c|cc}
 &  & \textsf{o} & \textsf{x} & \tabularnewline
\cline{2-4} 
 &  & \textsf{o} &  & \tabularnewline
\cline{2-4} 
 & \textsf{o} & \textsf{x} & \textsf{x} & \tabularnewline
\multicolumn{5}{c}{{\footnotesize{}(A)}}\tabularnewline
\end{tabular}\textsf{$\xrightarrow{o}$}%
\begin{tabular}{cc|c|cc}
 &  &  & \textsf{x} & \tabularnewline
\cline{2-4} 
 &  & \textsf{o} & \textsf{o} & \tabularnewline
\cline{2-4} 
 & \textsf{o} & \textsf{x} & \textsf{x} & \tabularnewline
\multicolumn{5}{c}{}\tabularnewline
\end{tabular}\textsf{$\xrightarrow{x}$}%
\begin{tabular}{cc|c|cc}
 &  &  & \textsf{x} & \tabularnewline
\cline{2-4} 
 & \textsf{x} & \textsf{o} & \textsf{o} & \tabularnewline
\cline{2-4} 
 & \textsf{o} &  & \textsf{x} & \tabularnewline
\multicolumn{5}{c}{}\tabularnewline
\end{tabular}\textsf{ }which is the original position.

Similarly, for game (B) \bigskip

\textsf{}%
\begin{tabular}{cc|c|cc}
 &  & \textsf{o} & \textsf{x} & \tabularnewline
\cline{2-4} 
 & \textsf{x} & \textsf{o} &  & \tabularnewline
\cline{2-4} 
 & \textsf{o} & \textsf{x} &  & \tabularnewline
\multicolumn{5}{c}{{\footnotesize{}(B)}}\tabularnewline
\end{tabular}\textsf{$\xrightarrow{o}$}%
\begin{tabular}{cc|c|cc}
 &  &  & \textsf{x} & \tabularnewline
\cline{2-4} 
 & \textsf{x} & \textsf{o} & \textsf{o} & \tabularnewline
\cline{2-4} 
 & \textsf{o} & \textsf{x} &  & \tabularnewline
\multicolumn{5}{c}{}\tabularnewline
\end{tabular}\textsf{$\xrightarrow{x}$ }%
\begin{tabular}{cc|c|cc}
 &  &  & \textsf{x} & \tabularnewline
\cline{2-4} 
 & \textsf{x} & \textsf{o} & \textsf{o} & \tabularnewline
\cline{2-4} 
 & \textsf{o} &  & \textsf{x} & \tabularnewline
\multicolumn{5}{c}{}\tabularnewline
\end{tabular}\textsf{ }\bigskip

which is again the original position.\bigskip\end{proof}
\begin{lem}
\label{lem:Xleavesthecenter}If X is to move from Loop position, and
it leaves the center, then X cannot win.\end{lem}
\begin{proof}
Suppose first that X leaves the center in the first move. There are
the following possibilities \bigskip

\textsf{}%
\begin{tabular}{cc|c|cc}
 &  & \textsf{x} & \textsf{o} & \tabularnewline
\cline{2-4} 
 & \textsf{o} &  & \textsf{x} & \tabularnewline
\cline{2-4} 
 & \textsf{x} &  & \textsf{o} & \tabularnewline
\multicolumn{5}{c}{{\footnotesize{}(A)}}\tabularnewline
\end{tabular}\textsf{}%
\begin{tabular}{cc|c|cc}
 &  &  & \textsf{o} & \tabularnewline
\cline{2-4} 
 & \textsf{o} &  & \textsf{x} & \tabularnewline
\cline{2-4} 
 & \textsf{x} & \textsf{x} & \textsf{o} & \tabularnewline
\multicolumn{5}{c}{{\footnotesize{}(B)}}\tabularnewline
\end{tabular}\textsf{}%
\begin{tabular}{cc|c|cc}
 & \textsf{x} &  & \textsf{o} & \tabularnewline
\cline{2-4} 
 & \textsf{o} &  & \textsf{x} & \tabularnewline
\cline{2-4} 
 & \textsf{x} &  & \textsf{o} & \tabularnewline
\multicolumn{5}{c}{{\footnotesize{}(C)}}\tabularnewline
\end{tabular}\bigskip

For game (A), O can force the sequence\bigskip

\textsf{}%
\begin{tabular}{cc|c|cc}
 &  & \textsf{x} & \textsf{o} & \tabularnewline
\cline{2-4} 
 & \textsf{o} &  & \textsf{x} & \tabularnewline
\cline{2-4} 
 & \textsf{x} &  & \textsf{o} & \tabularnewline
\multicolumn{5}{c}{{\footnotesize{}(A)}}\tabularnewline
\end{tabular}\textsf{$\xrightarrow{o}$}%
\begin{tabular}{cc|c|cc}
 &  & \textsf{x} &  & \tabularnewline
\cline{2-4} 
 & \textsf{o} & \textsf{o} & \textsf{x} & \tabularnewline
\cline{2-4} 
 & \textsf{x} &  & \textsf{o} & \tabularnewline
\multicolumn{5}{c}{}\tabularnewline
\end{tabular}\textsf{$\xrightarrow{x}$}%
\begin{tabular}{cc|c|cc}
 & \textsf{x} &  &  & \tabularnewline
\cline{2-4} 
 & \textsf{o} & \textsf{o} & \textsf{x} & \tabularnewline
\cline{2-4} 
 & \textsf{x} &  & \textsf{o} & \tabularnewline
\multicolumn{5}{c}{}\tabularnewline
\end{tabular}\bigskip

Now by Lemma \ref{lem:LemmaOloop}, player X cannot win. If X creates
game (B), then it Loses since O can do\bigskip

\textsf{}%
\begin{tabular}{cc|c|cc}
 &  &  & \textsf{o} & \tabularnewline
\cline{2-4} 
 & \textsf{o} &  & \textsf{x} & \tabularnewline
\cline{2-4} 
 & \textsf{x} & \textsf{x} & \textsf{o} & \tabularnewline
\multicolumn{5}{c}{{\footnotesize{}(B)}}\tabularnewline
\end{tabular}\textsf{$\xrightarrow{o}$}%
\begin{tabular}{cc|c|cc}
 &  &  &  & \tabularnewline
\cline{2-4} 
 & \textsf{o} & \textsf{o} & \textsf{x} & \tabularnewline
\cline{2-4} 
 & \textsf{x} & \textsf{x} & \textsf{o} & \tabularnewline
\multicolumn{5}{c}{}\tabularnewline
\end{tabular}\textsf{$\xrightarrow{x}$}%
\begin{tabular}{cc|c|cc}
 &  &  & \textsf{x} & \tabularnewline
\cline{2-4} 
 & \textsf{o} & \textsf{o} &  & \tabularnewline
\cline{2-4} 
 & \textsf{x} & \textsf{x} & \textsf{o} & \tabularnewline
\multicolumn{5}{c}{}\tabularnewline
\end{tabular}\textsf{$\xrightarrow{o}$}%
\begin{tabular}{cc|c|cc}
 &  &  & \textsf{x} & \tabularnewline
\cline{2-4} 
 & \textsf{o} & \textsf{o} & \textsf{o} & \tabularnewline
\cline{2-4} 
 & \textsf{x} & \textsf{x} &  & \tabularnewline
\multicolumn{5}{c}{}\tabularnewline
\end{tabular} \bigskip

So game (B) does not happen. For game (C), player O moves\bigskip

\textsf{}%
\begin{tabular}{cc|c|cc}
 & \textsf{x} &  & \textsf{o} & \tabularnewline
\cline{2-4} 
 & \textsf{o} &  & \textsf{x} & \tabularnewline
\cline{2-4} 
 & \textsf{x} &  & \textsf{o} & \tabularnewline
\multicolumn{5}{c}{{\footnotesize{}(C)}}\tabularnewline
\end{tabular}\textsf{$\xrightarrow{o}$}%
\begin{tabular}{cc|c|cc}
 & \textsf{x} &  & \textsf{o} & \tabularnewline
\cline{2-4} 
 &  & \textsf{o} & \textsf{x} & \tabularnewline
\cline{2-4} 
 & \textsf{x} &  & \textsf{o} & \tabularnewline
\multicolumn{5}{c}{}\tabularnewline
\end{tabular}\textsf{ }\bigskip

If player X moves \textsf{$\xrightarrow{x}$}%
\begin{tabular}{cc|c|cc}
 &  &  & \textsf{o} & \tabularnewline
\cline{2-4} 
 & \textsf{x} & \textsf{o} & \textsf{x} & \tabularnewline
\cline{2-4} 
 & \textsf{x} &  & \textsf{o} & \tabularnewline
\multicolumn{5}{c}{}\tabularnewline
\end{tabular}\textsf{$\xrightarrow{o}$}%
\begin{tabular}{cc|c|cc}
 &  & \textsf{o} &  & \tabularnewline
\cline{2-4} 
 & \textsf{x} & \textsf{o} & \textsf{x} & \tabularnewline
\cline{2-4} 
 & \textsf{x} &  & \textsf{o} & \tabularnewline
\multicolumn{5}{c}{}\tabularnewline
\end{tabular} \bigskip

by which player O wins the game. So instead X plays\bigskip

\textsf{$\xrightarrow{x}$}%
\begin{tabular}{cc|c|cc}
 &  & \textsf{x} & \textsf{o} & \tabularnewline
\cline{2-4} 
 &  & \textsf{o} & \textsf{x} & \tabularnewline
\cline{2-4} 
 & \textsf{x} &  & \textsf{o} & \tabularnewline
\multicolumn{5}{c}{}\tabularnewline
\end{tabular} and O responds by forcing X\bigskip

\textsf{$\xrightarrow{o}$}%
\begin{tabular}{cc|c|cc}
 &  & \textsf{x} & \textsf{o} & \tabularnewline
\cline{2-4} 
 &  & \textsf{o} & \textsf{x} & \tabularnewline
\cline{2-4} 
 & \textsf{x} & \textsf{o} &  & \tabularnewline
\multicolumn{5}{c}{}\tabularnewline
\end{tabular}\textsf{$\xrightarrow{x}$}%
\begin{tabular}{cc|c|cc}
 &  & \textsf{x} & \textsf{o} & \tabularnewline
\cline{2-4} 
 &  & \textsf{o} &  & \tabularnewline
\cline{2-4} 
 & \textsf{x} & \textsf{o} & \textsf{x} & \tabularnewline
\multicolumn{5}{c}{}\tabularnewline
\end{tabular} \bigskip

Now by Lemma \ref{lem:LemmaOloop}, player X cannot win.

Suppose next that player X leaves the center but not in the first
move. Now we use the games (A), (B) and (C) from Lemma \ref{lem:Xholdsthecenter},
and note that the only missing case is that X leaves the center in
the second move after the Loop-position (this follows because otherwise
O has already returned to Loop). Case (A) is not in X's strategy,
and case (B) becomes instead

\textsf{}%
\begin{tabular}{cc|c|cc}
 &  & \textsf{x} & \textsf{o} & \tabularnewline
\cline{2-4} 
 & \textsf{o} & \textsf{x} &  & \tabularnewline
\cline{2-4} 
 & \textsf{x} &  & \textsf{o} & \tabularnewline
\multicolumn{5}{c}{{\footnotesize{}(B)}}\tabularnewline
\end{tabular}\textsf{$\xrightarrow{o}$}%
\begin{tabular}{cc|c|cc}
 &  & \textsf{x} & \textsf{o} & \tabularnewline
\cline{2-4} 
 &  & \textsf{x} &  & \tabularnewline
\cline{2-4} 
 & \textsf{x} & \textsf{o} & \textsf{o} & \tabularnewline
\multicolumn{5}{c}{}\tabularnewline
\end{tabular}\textsf{$\xrightarrow{x}$}%
\begin{tabular}{cc|c|cc}
 &  & \textsf{x} & \textsf{o} & \tabularnewline
\cline{2-4} 
 &  &  & \textsf{x} & \tabularnewline
\cline{2-4} 
 & \textsf{x} & \textsf{o} & \textsf{o} & \tabularnewline
\multicolumn{5}{c}{}\tabularnewline
\end{tabular}\textsf{$\xrightarrow{o}$}%
\begin{tabular}{cc|c|cc}
 &  & \textsf{x} & \textsf{o} & \tabularnewline
\cline{2-4} 
 &  & \textsf{o} & \textsf{x} & \tabularnewline
\cline{2-4} 
 & \textsf{x} & \textsf{o} &  & \tabularnewline
\multicolumn{5}{c}{}\tabularnewline
\end{tabular}\textsf{$\xrightarrow{x}$}%
\begin{tabular}{cc|c|cc}
 &  & \textsf{x} & \textsf{o} & \tabularnewline
\cline{2-4} 
 &  & \textsf{o} &  & \tabularnewline
\cline{2-4} 
 & \textsf{x} & \textsf{o} & \textsf{x} & \tabularnewline
\multicolumn{5}{c}{}\tabularnewline
\end{tabular}\bigskip

from which player O has a non-losing strategy by Lemma \ref{lem:LemmaOloop}.
In case (C), we get \bigskip

\textsf{}%
\begin{tabular}{cc|c|cc}
 &  &  & \textsf{o} & \tabularnewline
\cline{2-4} 
 & \textsf{o} & \textsf{x} & \textsf{x} & \tabularnewline
\cline{2-4} 
 &  & \textsf{x} & \textsf{o} & \tabularnewline
\multicolumn{5}{c}{{\footnotesize{}(C)}}\tabularnewline
\end{tabular}\textsf{$\xrightarrow{o}$}%
\begin{tabular}{cc|c|cc}
 &  & \textsf{o} &  & \tabularnewline
\cline{2-4} 
 & \textsf{o} & \textsf{x} & \textsf{x} & \tabularnewline
\cline{2-4} 
 &  & \textsf{x} & \textsf{o} & \tabularnewline
\multicolumn{5}{c}{}\tabularnewline
\end{tabular}\textsf{$\xrightarrow{x}$}%
\begin{tabular}{cc|c|cc}
 & \textsf{x} & \textsf{o} &  & \tabularnewline
\cline{2-4} 
 & \textsf{o} &  & \textsf{x} & \tabularnewline
\cline{2-4} 
 &  & \textsf{x} & \textsf{o} & \tabularnewline
\multicolumn{5}{c}{}\tabularnewline
\end{tabular}\bigskip

\textsf{or} \bigskip

\textsf{}%
\begin{tabular}{cc|c|cc}
 &  &  & \textsf{o} & \tabularnewline
\cline{2-4} 
 & \textsf{o} & \textsf{x} & \textsf{x} & \tabularnewline
\cline{2-4} 
 &  & \textsf{x} & \textsf{o} & \tabularnewline
\multicolumn{5}{c}{{\footnotesize{}(C)}}\tabularnewline
\end{tabular}\textsf{$\xrightarrow{o}$}%
\begin{tabular}{cc|c|cc}
 &  & \textsf{o} &  & \tabularnewline
\cline{2-4} 
 & \textsf{o} & \textsf{x} & \textsf{x} & \tabularnewline
\cline{2-4} 
 &  & \textsf{x} & \textsf{o} & \tabularnewline
\multicolumn{5}{c}{}\tabularnewline
\end{tabular}\textsf{$\xrightarrow{x}$}%
\begin{tabular}{cc|c|cc}
 &  & \textsf{o} & \textsf{x} & \tabularnewline
\cline{2-4} 
 & \textsf{o} &  & \textsf{x} & \tabularnewline
\cline{2-4} 
 &  & \textsf{x} & \textsf{o} & \tabularnewline
\multicolumn{5}{c}{}\tabularnewline
\end{tabular} \bigskip

In the first case O creates a Zugzwang by moving into the center,
and in the second case, O can trap X.\bigskip\end{proof}
\begin{thm}
\label{thm:XLoop}Player X cannot win moving from a Loop position.\end{thm}
\begin{proof}
The proof follows from Lemmas \ref{lem:Xholdsthecenter} and \ref{lem:Xleavesthecenter}.
\end{proof}

\subsection{Games with two stones on the board}

There are three positions with exactly two stones modulo symmetries.
We begin by ruling out that player O gets to start in the center.\bigskip
\begin{lem}
\label{thm:OmiddleXcorner}For game\textsf{}%
\begin{tabular}{cc|c|cc}
 &  &  & \textsf{x} & \tabularnewline
\cline{2-4} 
 &  & \textsf{o} &  & \tabularnewline
\cline{2-4} 
 &  &  &  & \tabularnewline
\multicolumn{5}{c}{}\tabularnewline
\end{tabular}\textsf{ }X cannot win.\bigskip\end{lem}
\begin{proof}
We have the following possibilities for player X
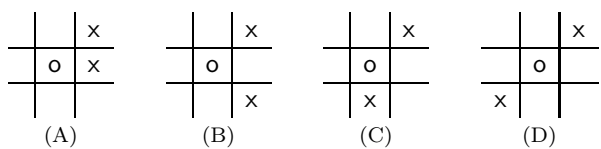
\bigskip

\textsf{}%
\begin{tabular}{cc|c|cc}
 &  &  & \textsf{x} & \tabularnewline
\cline{2-4} 
 &  & \textsf{o} & \textsf{x} & \tabularnewline
\cline{2-4} 
 &  &  &  & \tabularnewline
\multicolumn{5}{c}{{\footnotesize{}(A)}}\tabularnewline
\end{tabular}\textsf{}%
\begin{tabular}{cc|c|cc}
 &  &  & \textsf{x} & \tabularnewline
\cline{2-4} 
 &  & \textsf{o} &  & \tabularnewline
\cline{2-4} 
 &  &  & \textsf{x} & \tabularnewline
\multicolumn{5}{c}{{\footnotesize{}(B)}}\tabularnewline
\end{tabular}\textsf{}%
\begin{tabular}{cc|c|cc}
 &  &  & \textsf{x} & \tabularnewline
\cline{2-4} 
 &  & \textsf{o} &  & \tabularnewline
\cline{2-4} 
 &  & \textsf{x} &  & \tabularnewline
\multicolumn{5}{c}{{\footnotesize{}(C)}}\tabularnewline
\end{tabular}\textsf{}%
\begin{tabular}{cc|c|cc}
 &  &  & \textsf{x} & \tabularnewline
\cline{2-4} 
 &  & \textsf{o} &  & \tabularnewline
\cline{2-4} 
 & \textsf{x} &  &  & \tabularnewline
\multicolumn{5}{c}{{\footnotesize{}(D)}}\tabularnewline
\end{tabular}\bigskip

For game (A), player O can force a Loop by\bigskip

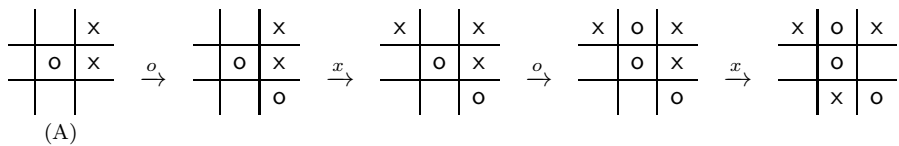
\textsf{}%
\begin{tabular}{cc|c|cc}
 &  &  & \textsf{x} & \tabularnewline
\cline{2-4} 
 &  & \textsf{o} & \textsf{x} & \tabularnewline
\cline{2-4} 
 &  &  &  & \tabularnewline
\multicolumn{5}{c}{{\footnotesize{}(A)}}\tabularnewline
\end{tabular}\textsf{$\xrightarrow{o}$}%
\begin{tabular}{cc|c|cc}
 &  &  & \textsf{x} & \tabularnewline
\cline{2-4} 
 &  & \textsf{o} & \textsf{x} & \tabularnewline
\cline{2-4} 
 &  &  & \textsf{o} & \tabularnewline
\multicolumn{5}{c}{}\tabularnewline
\end{tabular}\textsf{$\xrightarrow{x}$}%
\begin{tabular}{cc|c|cc}
 & \textsf{x} &  & \textsf{x} & \tabularnewline
\cline{2-4} 
 &  & \textsf{o} & \textsf{x} & \tabularnewline
\cline{2-4} 
 &  &  & \textsf{o} & \tabularnewline
\multicolumn{5}{c}{}\tabularnewline
\end{tabular}\textsf{$\xrightarrow{o}$}%
\begin{tabular}{cc|c|cc}
 & \textsf{x} & \textsf{o} & \textsf{x} & \tabularnewline
\cline{2-4} 
 &  & \textsf{o} & \textsf{x} & \tabularnewline
\cline{2-4} 
 &  &  & \textsf{o} & \tabularnewline
\multicolumn{5}{c}{}\tabularnewline
\end{tabular}\textsf{$\xrightarrow{x}$}%
\begin{tabular}{cc|c|cc}
 & \textsf{x} & \textsf{o} & \textsf{x} & \tabularnewline
\cline{2-4} 
 &  & \textsf{o} &  & \tabularnewline
\cline{2-4} 
 &  & \textsf{x} & \textsf{o} & \tabularnewline
\multicolumn{5}{c}{}\tabularnewline
\end{tabular}\bigskip

For game (B), player X must respond\bigskip

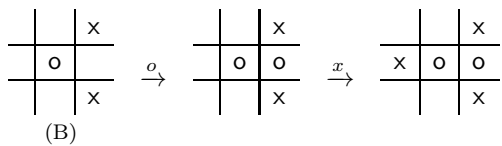
\textsf{}%
\begin{tabular}{cc|c|cc}
 &  &  & \textsf{x} & \tabularnewline
\cline{2-4} 
 &  & \textsf{o} &  & \tabularnewline
\cline{2-4} 
 &  &  & \textsf{x} & \tabularnewline
\multicolumn{5}{c}{{\footnotesize{}(B)}}\tabularnewline
\end{tabular}\textsf{$\xrightarrow{o}$}%
\begin{tabular}{cc|c|cc}
 &  &  & \textsf{x} & \tabularnewline
\cline{2-4} 
 &  & \textsf{o} & \textsf{o} & \tabularnewline
\cline{2-4} 
 &  &  & \textsf{x} & \tabularnewline
\multicolumn{5}{c}{}\tabularnewline
\end{tabular}\textsf{$\xrightarrow{x}$}%
\begin{tabular}{cc|c|cc}
 &  &  & \textsf{x} & \tabularnewline
\cline{2-4} 
 & \textsf{x} & \textsf{o} & \textsf{o} & \tabularnewline
\cline{2-4} 
 &  &  & \textsf{x} & \tabularnewline
\multicolumn{5}{c}{}\tabularnewline
\end{tabular} \bigskip

and now player O can draw the game by\bigskip

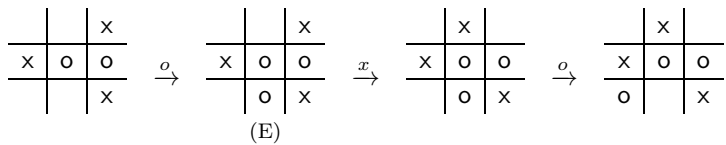
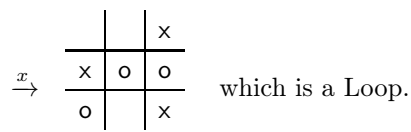
\textsf{}%
\begin{tabular}{cc|c|cc}
 &  &  & \textsf{x} & \tabularnewline
\cline{2-4} 
 & \textsf{x} & \textsf{o} & \textsf{o} & \tabularnewline
\cline{2-4} 
 &  &  & \textsf{x} & \tabularnewline
\multicolumn{5}{c}{}\tabularnewline
\end{tabular}\textsf{$\xrightarrow{o}$}%
\begin{tabular}{cc|c|cc}
 &  &  & \textsf{x} & \tabularnewline
\cline{2-4} 
 & \textsf{x} & \textsf{o} & \textsf{o} & \tabularnewline
\cline{2-4} 
 &  & \textsf{o} & \textsf{x} & \tabularnewline
\multicolumn{5}{c}{{\footnotesize{}(E)}}\tabularnewline
\end{tabular}\textsf{$\xrightarrow{x}$}%
\begin{tabular}{cc|c|cc}
 &  & \textsf{x} &  & \tabularnewline
\cline{2-4} 
 & \textsf{x} & \textsf{o} & \textsf{o} & \tabularnewline
\cline{2-4} 
 &  & \textsf{o} & \textsf{x} & \tabularnewline
\multicolumn{5}{c}{}\tabularnewline
\end{tabular}\textsf{$\xrightarrow{o}$}%
\begin{tabular}{cc|c|cc}
 &  & \textsf{x} &  & \tabularnewline
\cline{2-4} 
 & \textsf{x} & \textsf{o} & \textsf{o} & \tabularnewline
\cline{2-4} 
 & \textsf{o} &  & \textsf{x} & \tabularnewline
\multicolumn{5}{c}{}\tabularnewline
\end{tabular}\bigskip

\begin{onehalfspace}
\noindent \textsf{$\xrightarrow{x}$}%
\begin{tabular}{cc|c|cc}
 &  &  & \textsf{x} & \tabularnewline
\cline{2-4} 
 & \textsf{x} & \textsf{o} & \textsf{o} & \tabularnewline
\cline{2-4} 
 & \textsf{o} &  & \textsf{x} & \tabularnewline
\multicolumn{5}{c}{}\tabularnewline
\end{tabular} which is a Loop.
\end{onehalfspace}

For game (C), if player O forces the sequence of positions \bigskip

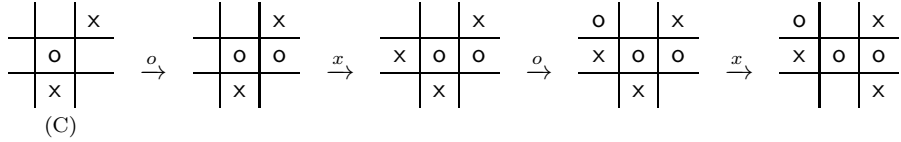
\textsf{}%
\begin{tabular}{cc|c|cc}
 &  &  & \textsf{x} & \tabularnewline
\cline{2-4} 
 &  & \textsf{o} &  & \tabularnewline
\cline{2-4} 
 &  & \textsf{x} &  & \tabularnewline
\multicolumn{5}{c}{{\footnotesize{}(C)}}\tabularnewline
\end{tabular}\textsf{$\xrightarrow{o}$}%
\begin{tabular}{cc|c|cc}
 &  &  & \textsf{x} & \tabularnewline
\cline{2-4} 
 &  & \textsf{o} & \textsf{o} & \tabularnewline
\cline{2-4} 
 &  & \textsf{x} &  & \tabularnewline
\multicolumn{5}{c}{}\tabularnewline
\end{tabular}\textsf{$\xrightarrow{x}$}%
\begin{tabular}{cc|c|cc}
 &  &  & \textsf{x} & \tabularnewline
\cline{2-4} 
 & \textsf{x} & \textsf{o} & \textsf{o} & \tabularnewline
\cline{2-4} 
 &  & \textsf{x} &  & \tabularnewline
\multicolumn{5}{c}{}\tabularnewline
\end{tabular}\textsf{$\xrightarrow{o}$}%
\begin{tabular}{cc|c|cc}
 & \textsf{o} &  & \textsf{x} & \tabularnewline
\cline{2-4} 
 & \textsf{x} & \textsf{o} & \textsf{o} & \tabularnewline
\cline{2-4} 
 &  & \textsf{x} &  & \tabularnewline
\multicolumn{5}{c}{}\tabularnewline
\end{tabular}\textsf{$\xrightarrow{x}$}%
\begin{tabular}{cc|c|cc}
 & \textsf{o} &  & \textsf{x} & \tabularnewline
\cline{2-4} 
 & \textsf{x} & \textsf{o} & \textsf{o} & \tabularnewline
\cline{2-4} 
 &  &  & \textsf{x} & \tabularnewline
\multicolumn{5}{c}{}\tabularnewline
\end{tabular} \bigskip

then, by Lemma \ref{lem:LemmaOloop} $X$ cannot win. Similarly, for
game (D), if O forces the sequence \bigskip

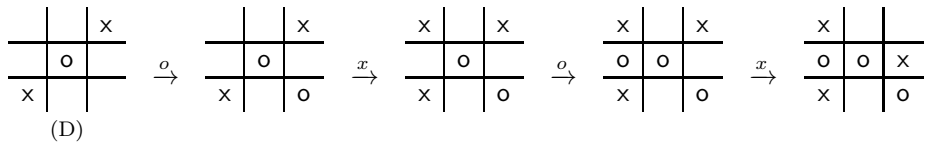
\textsf{}%
\begin{tabular}{cc|c|cc}
 &  &  & \textsf{x} & \tabularnewline
\cline{2-4} 
 &  & \textsf{o} &  & \tabularnewline
\cline{2-4} 
 & \textsf{x} &  &  & \tabularnewline
\multicolumn{5}{c}{{\footnotesize{}(D)}}\tabularnewline
\end{tabular}\textsf{$\xrightarrow{o}$}%
\begin{tabular}{cc|c|cc}
 &  &  & \textsf{x} & \tabularnewline
\cline{2-4} 
 &  & \textsf{o} &  & \tabularnewline
\cline{2-4} 
 & \textsf{x} &  & \textsf{o} & \tabularnewline
\multicolumn{5}{c}{}\tabularnewline
\end{tabular}\textsf{$\xrightarrow{x}$}%
\begin{tabular}{cc|c|cc}
 & \textsf{x} &  & \textsf{x} & \tabularnewline
\cline{2-4} 
 &  & \textsf{o} &  & \tabularnewline
\cline{2-4} 
 & \textsf{x} &  & \textsf{o} & \tabularnewline
\multicolumn{5}{c}{}\tabularnewline
\end{tabular}\textsf{$\xrightarrow{o}$}%
\begin{tabular}{cc|c|cc}
 & \textsf{x} &  & \textsf{x} & \tabularnewline
\cline{2-4} 
 & \textsf{o} & \textsf{o} &  & \tabularnewline
\cline{2-4} 
 & \textsf{x} &  & \textsf{o} & \tabularnewline
\multicolumn{5}{c}{}\tabularnewline
\end{tabular}\textsf{$\xrightarrow{x}$}%
\begin{tabular}{cc|c|cc}
 & \textsf{x} &  &  & \tabularnewline
\cline{2-4} 
 & \textsf{o} & \textsf{o} & \textsf{x} & \tabularnewline
\cline{2-4} 
 & \textsf{x} &  & \textsf{o} & \tabularnewline
\multicolumn{5}{c}{}\tabularnewline
\end{tabular}\bigskip

Then by Lemma \ref{lem:LemmaOloop} player X cannot win.
\end{proof}
\bigskip
\begin{lem}
\label{lem:secondOcenter}For game\textsf{ }%
\begin{tabular}{cc|c|cc}
 &  &  &  & \tabularnewline
\cline{2-4} 
 &  & \textsf{o} & \textsf{x} & \tabularnewline
\cline{2-4} 
 &  &  &  & \tabularnewline
\multicolumn{5}{c}{}\tabularnewline
\end{tabular}\textsf{ }X cannot win.\end{lem}
\begin{proof}
Here player X has the possibilities \bigskip

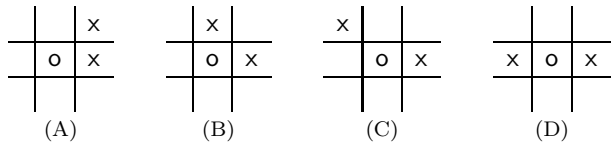
\textsf{}%
\begin{tabular}{cc|c|cc}
 &  &  & \textsf{x} & \tabularnewline
\cline{2-4} 
 &  & \textsf{o} & \textsf{x} & \tabularnewline
\cline{2-4} 
 &  &  &  & \tabularnewline
\multicolumn{5}{c}{{\footnotesize{}(A)}}\tabularnewline
\end{tabular}\textsf{}%
\begin{tabular}{cc|c|cc}
 &  & \textsf{x} &  & \tabularnewline
\cline{2-4} 
 &  & \textsf{o} & \textsf{x} & \tabularnewline
\cline{2-4} 
 &  &  &  & \tabularnewline
\multicolumn{5}{c}{{\footnotesize{}(B)}}\tabularnewline
\end{tabular}\textsf{}%
\begin{tabular}{cc|c|cc}
 & \textsf{x} &  &  & \tabularnewline
\cline{2-4} 
 &  & \textsf{o} & \textsf{x} & \tabularnewline
\cline{2-4} 
 &  &  &  & \tabularnewline
\multicolumn{5}{c}{{\footnotesize{}(C)}}\tabularnewline
\end{tabular}\textsf{}%
\begin{tabular}{cc|c|cc}
 &  &  &  & \tabularnewline
\cline{2-4} 
 & \textsf{x} & \textsf{o} & \textsf{x} & \tabularnewline
\cline{2-4} 
 &  &  &  & \tabularnewline
\multicolumn{5}{c}{{\footnotesize{}(D)}}\tabularnewline
\end{tabular} \bigskip

Games (A) and (C) are the same games as (A) and (C) in the proof of
Theorem \ref{thm:OmiddleXcorner}. 

For game (B), player O can force X's moves by\bigskip

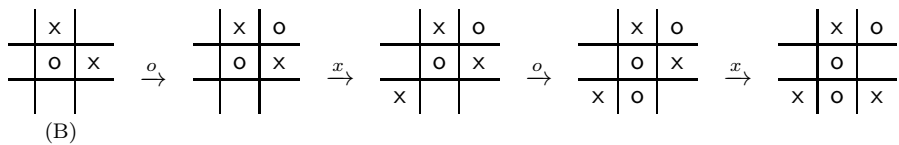
\textsf{}%
\begin{tabular}{cc|c|cc}
 &  & \textsf{x} &  & \tabularnewline
\cline{2-4} 
 &  & \textsf{o} & \textsf{x} & \tabularnewline
\cline{2-4} 
 &  &  &  & \tabularnewline
\multicolumn{5}{c}{{\footnotesize{}(B)}}\tabularnewline
\end{tabular}\textsf{$\xrightarrow{o}$}%
\begin{tabular}{cc|c|cc}
 &  & \textsf{x} & \textsf{o} & \tabularnewline
\cline{2-4} 
 &  & \textsf{o} & \textsf{x} & \tabularnewline
\cline{2-4} 
 &  &  &  & \tabularnewline
\multicolumn{5}{c}{}\tabularnewline
\end{tabular}\textsf{$\xrightarrow{x}$}%
\begin{tabular}{cc|c|cc}
 &  & \textsf{x} & \textsf{o} & \tabularnewline
\cline{2-4} 
 &  & \textsf{o} & \textsf{x} & \tabularnewline
\cline{2-4} 
 & \textsf{x} &  &  & \tabularnewline
\multicolumn{5}{c}{}\tabularnewline
\end{tabular}\textsf{$\xrightarrow{o}$}%
\begin{tabular}{cc|c|cc}
 &  & \textsf{x} & \textsf{o} & \tabularnewline
\cline{2-4} 
 &  & \textsf{o} & \textsf{x} & \tabularnewline
\cline{2-4} 
 & \textsf{x} & \textsf{o} &  & \tabularnewline
\multicolumn{5}{c}{}\tabularnewline
\end{tabular}\textsf{$\xrightarrow{x}$}%
\begin{tabular}{cc|c|cc}
 &  & \textsf{x} & \textsf{o} & \tabularnewline
\cline{2-4} 
 &  & \textsf{o} &  & \tabularnewline
\cline{2-4} 
 & \textsf{x} & \textsf{o} & \textsf{x} & \tabularnewline
\multicolumn{5}{c}{}\tabularnewline
\end{tabular}\bigskip\textsf{ }

Then by Lemma \ref{lem:LemmaOloop}, X cannot win. For game (D), player
O can force X's moves by\bigskip

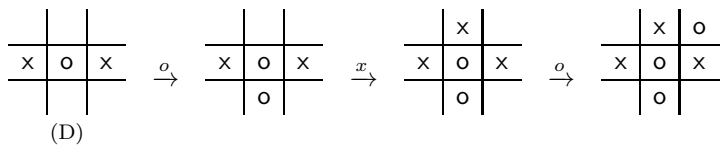
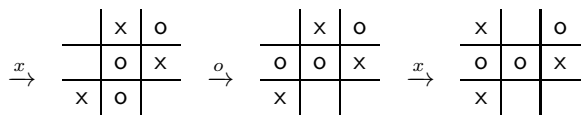
\textsf{}%
\begin{tabular}{cc|c|cc}
 &  &  &  & \tabularnewline
\cline{2-4} 
 & \textsf{x} & \textsf{o} & \textsf{x} & \tabularnewline
\cline{2-4} 
 &  &  &  & \tabularnewline
\multicolumn{5}{c}{{\footnotesize{}(D)}}\tabularnewline
\end{tabular}\textsf{$\xrightarrow{o}$}%
\begin{tabular}{cc|c|cc}
 &  &  &  & \tabularnewline
\cline{2-4} 
 & \textsf{x} & \textsf{o} & \textsf{x} & \tabularnewline
\cline{2-4} 
 &  & \textsf{o} &  & \tabularnewline
\multicolumn{5}{c}{}\tabularnewline
\end{tabular}\textsf{$\xrightarrow{x}$}%
\begin{tabular}{cc|c|cc}
 &  & \textsf{x} &  & \tabularnewline
\cline{2-4} 
 & \textsf{x} & \textsf{o} & \textsf{x} & \tabularnewline
\cline{2-4} 
 &  & \textsf{o} &  & \tabularnewline
\multicolumn{5}{c}{}\tabularnewline
\end{tabular}\textsf{$\xrightarrow{o}$}%
\begin{tabular}{cc|c|cc}
 &  & \textsf{x} & \textsf{o} & \tabularnewline
\cline{2-4} 
 & \textsf{x} & \textsf{o} & \textsf{x} & \tabularnewline
\cline{2-4} 
 &  & \textsf{o} &  & \tabularnewline
\multicolumn{5}{c}{}\tabularnewline
\end{tabular}\bigskip

\textsf{$\xrightarrow{x}$}%
\begin{tabular}{cc|c|cc}
 &  & \textsf{x} & \textsf{o} & \tabularnewline
\cline{2-4} 
 &  & \textsf{o} & \textsf{x} & \tabularnewline
\cline{2-4} 
 & \textsf{x} & \textsf{o} &  & \tabularnewline
\multicolumn{5}{c}{}\tabularnewline
\end{tabular}\textsf{$\xrightarrow{o}$}%
\begin{tabular}{cc|c|cc}
 &  & \textsf{x} & \textsf{o} & \tabularnewline
\cline{2-4} 
 & \textsf{o} & \textsf{o} & \textsf{x} & \tabularnewline
\cline{2-4} 
 & \textsf{x} &  &  & \tabularnewline
\multicolumn{5}{c}{}\tabularnewline
\end{tabular}\textsf{$\xrightarrow{x}$}%
\begin{tabular}{cc|c|cc}
 & \textsf{x} &  & \textsf{o} & \tabularnewline
\cline{2-4} 
 & \textsf{o} & \textsf{o} & \textsf{x} & \tabularnewline
\cline{2-4} 
 & \textsf{x} &  &  & \tabularnewline
\multicolumn{5}{c}{}\tabularnewline
\end{tabular}\bigskip

Then, by Lemma \ref{lem:LemmaOloop}, player X cannot win.\bigskip
\end{proof}
The next result deals with the most delicate position. There are some
ideas that are very interesting, that did not appear so far.
\begin{lem}
\label{lem:delicate}For game\textsf{ }%
\begin{tabular}{cc|c|cc}
 &  &  & \textsf{o} & \tabularnewline
\cline{2-4} 
 &  & \textsf{x} &  & \tabularnewline
\cline{2-4} 
 &  &  &  & \tabularnewline
\multicolumn{5}{c}{}\tabularnewline
\end{tabular}\textsf{ }X cannot win.\end{lem}
\begin{proof}
We have the cases to check\bigskip

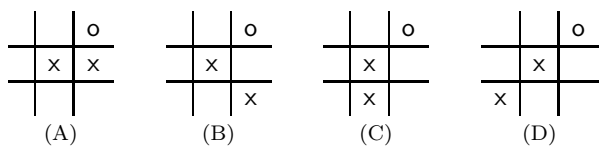
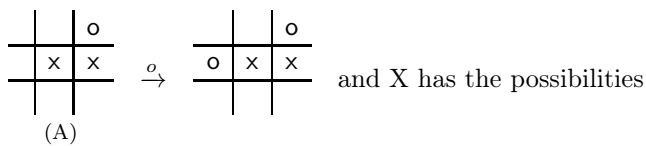
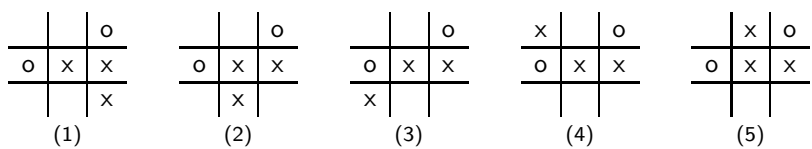
\textsf{}%
\begin{tabular}{cc|c|cc}
 &  &  & \textsf{o} & \tabularnewline
\cline{2-4} 
 &  & \textsf{x} & \textsf{x} & \tabularnewline
\cline{2-4} 
 &  &  &  & \tabularnewline
\multicolumn{5}{c}{{\footnotesize{}(A)}}\tabularnewline
\end{tabular}\textsf{}%
\begin{tabular}{cc|c|cc}
 &  &  & \textsf{o} & \tabularnewline
\cline{2-4} 
 &  & \textsf{x} &  & \tabularnewline
\cline{2-4} 
 &  &  & \textsf{x} & \tabularnewline
\multicolumn{5}{c}{{\footnotesize{}(B)}}\tabularnewline
\end{tabular}\textsf{}%
\begin{tabular}{cc|c|cc}
 &  &  & \textsf{o} & \tabularnewline
\cline{2-4} 
 &  & \textsf{x} &  & \tabularnewline
\cline{2-4} 
 &  & \textsf{x} &  & \tabularnewline
\multicolumn{5}{c}{{\footnotesize{}(C)}}\tabularnewline
\end{tabular}\textsf{}%
\begin{tabular}{cc|c|cc}
 &  &  & \textsf{o} & \tabularnewline
\cline{2-4} 
 &  & \textsf{x} &  & \tabularnewline
\cline{2-4} 
 & \textsf{x} &  &  & \tabularnewline
\multicolumn{5}{c}{{\footnotesize{}(D)}}\tabularnewline
\end{tabular}\bigskip

For position (A), \bigskip

\textsf{}%
\begin{tabular}{cc|c|cc}
 &  &  & \textsf{o} & \tabularnewline
\cline{2-4} 
 &  & \textsf{x} & \textsf{x} & \tabularnewline
\cline{2-4} 
 &  &  &  & \tabularnewline
\multicolumn{5}{c}{{\footnotesize{}(A)}}\tabularnewline
\end{tabular}\textsf{$\xrightarrow{o}$}%
\begin{tabular}{cc|c|cc}
 &  &  & \textsf{o} & \tabularnewline
\cline{2-4} 
 & \textsf{o} & \textsf{x} & \textsf{x} & \tabularnewline
\cline{2-4} 
 &  &  &  & \tabularnewline
\multicolumn{5}{c}{}\tabularnewline
\end{tabular}and X has the possibilities \bigskip

\textsf{}%
\begin{tabular}{cc|c|cc}
 &  &  & \textsf{o} & \tabularnewline
\cline{2-4} 
 & \textsf{o} & \textsf{x} & \textsf{x} & \tabularnewline
\cline{2-4} 
 &  &  & \textsf{x} & \tabularnewline
\multicolumn{5}{c}{\textsf{\footnotesize{}(1)}}\tabularnewline
\end{tabular}\textsf{}%
\begin{tabular}{cc|c|cc}
 &  &  & \textsf{o} & \tabularnewline
\cline{2-4} 
 & \textsf{o} & \textsf{x} & \textsf{x} & \tabularnewline
\cline{2-4} 
 &  & \textsf{x} &  & \tabularnewline
\multicolumn{5}{c}{\textsf{\footnotesize{}(2)}}\tabularnewline
\end{tabular}\textsf{}%
\begin{tabular}{cc|c|cc}
 &  &  & \textsf{o} & \tabularnewline
\cline{2-4} 
 & \textsf{o} & \textsf{x} & \textsf{x} & \tabularnewline
\cline{2-4} 
 & \textsf{x} &  &  & \tabularnewline
\multicolumn{5}{c}{\textsf{\footnotesize{}(3)}}\tabularnewline
\end{tabular}\textsf{}%
\begin{tabular}{cc|c|cc}
 & \textsf{x} &  & \textsf{o} & \tabularnewline
\cline{2-4} 
 & \textsf{o} & \textsf{x} & \textsf{x} & \tabularnewline
\cline{2-4} 
 &  &  &  & \tabularnewline
\multicolumn{5}{c}{\textsf{\footnotesize{}(4)}}\tabularnewline
\end{tabular}\textsf{}%
\begin{tabular}{cc|c|cc}
 &  & \textsf{x} & \textsf{o} & \tabularnewline
\cline{2-4} 
 & \textsf{o} & \textsf{x} & \textsf{x} & \tabularnewline
\cline{2-4} 
 &  &  &  & \tabularnewline
\multicolumn{5}{c}{\textsf{\footnotesize{}(5)}}\tabularnewline
\end{tabular}\bigskip

For position 1,\bigskip

\textsf{}%
\begin{tabular}{cc|c|cc}
 &  &  & \textsf{o} & \tabularnewline
\cline{2-4} 
 & \textsf{o} & \textsf{x} & \textsf{x} & \tabularnewline
\cline{2-4} 
 &  &  & \textsf{x} & \tabularnewline
\multicolumn{5}{c}{\textsf{\footnotesize{}(1)}}\tabularnewline
\end{tabular}\textsf{$\xrightarrow{o}$}%
\begin{tabular}{cc|c|cc}
 & \textsf{o} &  & \textsf{o} & \tabularnewline
\cline{2-4} 
 & \textsf{o} & \textsf{x} & \textsf{x} & \tabularnewline
\cline{2-4} 
 &  &  & \textsf{x} & \tabularnewline
\multicolumn{5}{c}{}\tabularnewline
\end{tabular}\bigskip

Now if \bigskip

\textsf{}%
\begin{tabular}{cc|c|cc}
 & \textsf{o} &  & \textsf{o} & \tabularnewline
\cline{2-4} 
 & \textsf{o} & \textsf{x} & \textsf{x} & \tabularnewline
\cline{2-4} 
 &  &  & \textsf{x} & \tabularnewline
\multicolumn{5}{c}{}\tabularnewline
\end{tabular}\textsf{$\xrightarrow{x}$}%
\begin{tabular}{cc|c|cc}
 & \textsf{o} & \textsf{x} & \textsf{o} & \tabularnewline
\cline{2-4} 
 & \textsf{o} &  & \textsf{x} & \tabularnewline
\cline{2-4} 
 &  &  & \textsf{x} & \tabularnewline
\multicolumn{5}{c}{}\tabularnewline
\end{tabular}\textsf{$\xrightarrow{o}$}%
\begin{tabular}{cc|c|cc}
 &  & \textsf{x} & \textsf{o} & \tabularnewline
\cline{2-4} 
 & \textsf{o} & \textsf{o} & \textsf{x} & \tabularnewline
\cline{2-4} 
 &  &  & \textsf{x} & \tabularnewline
\multicolumn{5}{c}{}\tabularnewline
\end{tabular}\textsf{ }and X cannot avoid a loss,\bigskip

and if\bigskip

\textsf{}%
\begin{tabular}{cc|c|cc}
 & \textsf{o} &  & \textsf{o} & \tabularnewline
\cline{2-4} 
 & \textsf{o} & \textsf{x} & \textsf{x} & \tabularnewline
\cline{2-4} 
 &  &  & \textsf{x} & \tabularnewline
\multicolumn{5}{c}{}\tabularnewline
\end{tabular}\textsf{$\xrightarrow{x}$}%
\begin{tabular}{cc|c|cc}
 & \textsf{o} & \textsf{x} & \textsf{o} & \tabularnewline
\cline{2-4} 
 & \textsf{o} & \textsf{x} &  & \tabularnewline
\cline{2-4} 
 &  &  & \textsf{x} & \tabularnewline
\multicolumn{5}{c}{}\tabularnewline
\end{tabular}\textsf{$\xrightarrow{o}$}%
\begin{tabular}{cc|c|cc}
 & \textsf{o} & \textsf{x} & \textsf{o} & \tabularnewline
\cline{2-4} 
 &  & \textsf{x} &  & \tabularnewline
\cline{2-4} 
 &  & \textsf{o} & \textsf{x} & \tabularnewline
\multicolumn{5}{c}{}\tabularnewline
\end{tabular}\textsf{ }player O creates a Loop.\bigskip

For position 2, \bigskip

\textsf{}%
\begin{tabular}{cc|c|cc}
 &  &  & \textsf{o} & \tabularnewline
\cline{2-4} 
 & \textsf{o} & \textsf{x} & \textsf{x} & \tabularnewline
\cline{2-4} 
 &  & \textsf{x} &  & \tabularnewline
\multicolumn{5}{c}{\textsf{\footnotesize{}(2)}}\tabularnewline
\end{tabular}\textsf{$\xrightarrow{o}$}%
\begin{tabular}{cc|c|cc}
 &  & \textsf{o} & \textsf{o} & \tabularnewline
\cline{2-4} 
 & \textsf{o} & \textsf{x} & \textsf{x} & \tabularnewline
\cline{2-4} 
 &  & \textsf{x} &  & \tabularnewline
\multicolumn{5}{c}{}\tabularnewline
\end{tabular}\textsf{$\xrightarrow{x}$}%
\begin{tabular}{cc|c|cc}
 & \textsf{x} & \textsf{o} & \textsf{o} & \tabularnewline
\cline{2-4} 
 & \textsf{o} &  & \textsf{x} & \tabularnewline
\cline{2-4} 
 &  & \textsf{x} &  & \tabularnewline
\multicolumn{5}{c}{}\tabularnewline
\end{tabular}\textsf{$\xrightarrow{o}$}%
\begin{tabular}{cc|c|cc}
 & \textsf{x} & \textsf{o} &  & \tabularnewline
\cline{2-4} 
 & \textsf{o} & \textsf{o} & \textsf{x} & \tabularnewline
\cline{2-4} 
 &  & \textsf{x} &  & \tabularnewline
\multicolumn{5}{c}{}\tabularnewline
\end{tabular}\textsf{ }\bigskip

which is a zugzwang, so no matter what player X does, it loses the
game. For position 3, player O creates a Loop by\bigskip

\textsf{}%
\begin{tabular}{cc|c|cc}
 &  &  & \textsf{o} & \tabularnewline
\cline{2-4} 
 & \textsf{o} & \textsf{x} & \textsf{x} & \tabularnewline
\cline{2-4} 
 & \textsf{x} &  &  & \tabularnewline
\multicolumn{5}{c}{\textsf{\footnotesize{}(3)}}\tabularnewline
\end{tabular}\textsf{$\xrightarrow{o}$}%
\begin{tabular}{cc|c|cc}
 &  &  & \textsf{o} & \tabularnewline
\cline{2-4} 
 & \textsf{o} & \textsf{x} & \textsf{x} & \tabularnewline
\cline{2-4} 
 & \textsf{x} &  & \textsf{o} & \tabularnewline
\multicolumn{5}{c}{}\tabularnewline
\end{tabular}\bigskip

Position 4 forces player O to create a Loop by\bigskip

\textsf{}%
\begin{tabular}{cc|c|cc}
 & \textsf{x} &  & \textsf{o} & \tabularnewline
\cline{2-4} 
 & \textsf{o} & \textsf{x} & \textsf{x} & \tabularnewline
\cline{2-4} 
 &  &  &  & \tabularnewline
\multicolumn{5}{c}{\textsf{\footnotesize{}(4)}}\tabularnewline
\end{tabular}\textsf{$\xrightarrow{o}$}%
\begin{tabular}{cc|c|cc}
 & \textsf{x} &  & \textsf{o} & \tabularnewline
\cline{2-4} 
 & \textsf{o} & \textsf{x} & \textsf{x} & \tabularnewline
\cline{2-4} 
 &  &  & \textsf{o} & \tabularnewline
\multicolumn{5}{c}{}\tabularnewline
\end{tabular}\bigskip

For position (5) \bigskip

\textsf{}%
\begin{tabular}{cc|c|cc}
 &  & \textsf{x} & \textsf{o} & \tabularnewline
\cline{2-4} 
 & \textsf{o} & \textsf{x} & \textsf{x} & \tabularnewline
\cline{2-4} 
 &  &  &  & \tabularnewline
\multicolumn{5}{c}{\textsf{\footnotesize{}(5)}}\tabularnewline
\end{tabular}\textsf{$\xrightarrow{o}$}%
\begin{tabular}{cc|c|cc}
 &  & \textsf{x} & \textsf{o} & \tabularnewline
\cline{2-4} 
 & \textsf{o} & \textsf{x} & \textsf{x} & \tabularnewline
\cline{2-4} 
 &  & \textsf{o} &  & \tabularnewline
\multicolumn{5}{c}{}\tabularnewline
\end{tabular}\textsf{ }\bigskip

and now player X has 3 options, disregarding symmetries\bigskip

\textsf{}%
\begin{tabular}{cc|c|cc}
 &  & \textsf{x} & \textsf{o} & \tabularnewline
\cline{2-4} 
 & \textsf{o} & \textsf{x} &  & \tabularnewline
\cline{2-4} 
 &  & \textsf{o} & \textsf{x} & \tabularnewline
\multicolumn{5}{c}{\textsf{\footnotesize{}(i)}}\tabularnewline
\end{tabular}\textsf{ }%
\begin{tabular}{cc|c|cc}
 &  & \textsf{x} & \textsf{o} & \tabularnewline
\cline{2-4} 
 & \textsf{o} &  & \textsf{x} & \tabularnewline
\cline{2-4} 
 &  & \textsf{o} & \textsf{x} & \tabularnewline
\multicolumn{5}{c}{\textsf{\footnotesize{}(ii)}}\tabularnewline
\end{tabular}\textsf{ }%
\begin{tabular}{cc|c|cc}
 &  & \textsf{x} & \textsf{o} & \tabularnewline
\cline{2-4} 
 & \textsf{o} &  & \textsf{x} & \tabularnewline
\cline{2-4} 
 & \textsf{x} & \textsf{o} &  & \tabularnewline
\multicolumn{5}{c}{\textsf{\footnotesize{}(iii)}}\tabularnewline
\end{tabular}\bigskip

For (i), player O is forced to create a Loop by\bigskip

\textsf{}%
\begin{tabular}{cc|c|cc}
 &  & \textsf{x} & \textsf{o} & \tabularnewline
\cline{2-4} 
 & \textsf{o} & \textsf{x} &  & \tabularnewline
\cline{2-4} 
 &  & \textsf{o} & \textsf{x} & \tabularnewline
\multicolumn{5}{c}{\textsf{\footnotesize{}(i)}}\tabularnewline
\end{tabular}\textsf{$\xrightarrow{o}$}%
\begin{tabular}{cc|c|cc}
 & \textsf{o} & \textsf{x} & \textsf{o} & \tabularnewline
\cline{2-4} 
 &  & \textsf{x} &  & \tabularnewline
\cline{2-4} 
 &  & \textsf{o} & \textsf{x} & \tabularnewline
\multicolumn{5}{c}{}\tabularnewline
\end{tabular}\bigskip

For (ii), player X is trapped by \bigskip

\textsf{}%
\begin{tabular}{cc|c|cc}
 &  & \textsf{x} & \textsf{o} & \tabularnewline
\cline{2-4} 
 & \textsf{o} &  & \textsf{x} & \tabularnewline
\cline{2-4} 
 &  & \textsf{o} & \textsf{x} & \tabularnewline
\multicolumn{5}{c}{\textsf{\footnotesize{}(ii)}}\tabularnewline
\end{tabular}\textsf{$\xrightarrow{o}$}%
\begin{tabular}{cc|c|cc}
 &  & \textsf{x} & \textsf{o} & \tabularnewline
\cline{2-4} 
 &  & \textsf{o} & \textsf{x} & \tabularnewline
\cline{2-4} 
 &  & \textsf{o} & \textsf{x} & \tabularnewline
\multicolumn{5}{c}{}\tabularnewline
\end{tabular}\textsf{$\xrightarrow{x}$}%
\begin{tabular}{cc|c|cc}
 & \textsf{x} &  & \textsf{o} & \tabularnewline
\cline{2-4} 
 &  & \textsf{o} & \textsf{x} & \tabularnewline
\cline{2-4} 
 &  & \textsf{o} & \textsf{x} & \tabularnewline
\multicolumn{5}{c}{}\tabularnewline
\end{tabular}\textsf{$\xrightarrow{o}$}%
\begin{tabular}{cc|c|cc}
 & \textsf{x} & \textsf{o} &  & \tabularnewline
\cline{2-4} 
 &  & \textsf{o} & \textsf{x} & \tabularnewline
\cline{2-4} 
 &  & \textsf{o} & \textsf{x} & \tabularnewline
\multicolumn{5}{c}{}\tabularnewline
\end{tabular}\bigskip

So X does not create game (ii). Finally, game (iii) allows player
O to win by a Zugzwang \bigskip

\textsf{}%
\begin{tabular}{cc|c|cc}
 &  & \textsf{x} & \textsf{o} & \tabularnewline
\cline{2-4} 
 & \textsf{o} &  & \textsf{x} & \tabularnewline
\cline{2-4} 
 & \textsf{x} & \textsf{o} &  & \tabularnewline
\multicolumn{5}{c}{\textsf{\footnotesize{}(iii)}}\tabularnewline
\end{tabular}\textsf{ $\xrightarrow{o}$}%
\begin{tabular}{cc|c|cc}
 &  & \textsf{x} &  & \tabularnewline
\cline{2-4} 
 & \textsf{o} & \textsf{o} & \textsf{x} & \tabularnewline
\cline{2-4} 
 & \textsf{x} & \textsf{o} &  & \tabularnewline
\multicolumn{5}{c}{}\tabularnewline
\end{tabular} \bigskip

Thus player X does not create game (iii). That finishes game (A).
For game (B), all moves are forced, which creates a Loop as follows
\bigskip

\textsf{}%
\begin{tabular}{cc|c|cc}
 &  &  & \textsf{o} & \tabularnewline
\cline{2-4} 
 &  & \textsf{x} &  & \tabularnewline
\cline{2-4} 
 &  &  & \textsf{x} & \tabularnewline
\multicolumn{5}{c}{{\footnotesize{}(B)}}\tabularnewline
\end{tabular}\textsf{$\xrightarrow{o}$}%
\begin{tabular}{cc|c|cc}
 & \textsf{o} &  & \textsf{o} & \tabularnewline
\cline{2-4} 
 &  & \textsf{x} &  & \tabularnewline
\cline{2-4} 
 &  &  & \textsf{x} & \tabularnewline
\multicolumn{5}{c}{}\tabularnewline
\end{tabular}\textsf{$\xrightarrow{x}$}%
\begin{tabular}{cc|c|cc}
 & \textsf{o} & \textsf{x} & \textsf{o} & \tabularnewline
\cline{2-4} 
 &  & \textsf{x} &  & \tabularnewline
\cline{2-4} 
 &  &  & \textsf{x} & \tabularnewline
\multicolumn{5}{c}{}\tabularnewline
\end{tabular}\textsf{$\xrightarrow{o}$}%
\begin{tabular}{cc|c|cc}
 & \textsf{o} & \textsf{x} & \textsf{o} & \tabularnewline
\cline{2-4} 
 &  & \textsf{x} &  & \tabularnewline
\cline{2-4} 
 &  & \textsf{o} & \textsf{x} & \tabularnewline
\multicolumn{5}{c}{}\tabularnewline
\end{tabular}\bigskip

For game (C), there are the following forced moves\bigskip

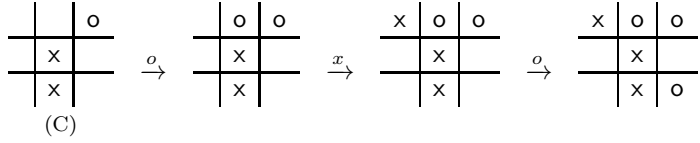
\textsf{}%
\begin{tabular}{cc|c|cc}
 &  &  & \textsf{o} & \tabularnewline
\cline{2-4} 
 &  & \textsf{x} &  & \tabularnewline
\cline{2-4} 
 &  & \textsf{x} &  & \tabularnewline
\multicolumn{5}{c}{{\footnotesize{}(C)}}\tabularnewline
\end{tabular}\textsf{$\xrightarrow{o}$}%
\begin{tabular}{cc|c|cc}
 &  & \textsf{o} & \textsf{o} & \tabularnewline
\cline{2-4} 
 &  & \textsf{x} &  & \tabularnewline
\cline{2-4} 
 &  & \textsf{x} &  & \tabularnewline
\multicolumn{5}{c}{}\tabularnewline
\end{tabular}\textsf{$\xrightarrow{x}$}%
\begin{tabular}{cc|c|cc}
 & \textsf{x} & \textsf{o} & \textsf{o} & \tabularnewline
\cline{2-4} 
 &  & \textsf{x} &  & \tabularnewline
\cline{2-4} 
 &  & \textsf{x} &  & \tabularnewline
\multicolumn{5}{c}{}\tabularnewline
\end{tabular}\textsf{$\xrightarrow{o}$}%
\begin{tabular}{cc|c|cc}
 & \textsf{x} & \textsf{o} & \textsf{o} & \tabularnewline
\cline{2-4} 
 &  & \textsf{x} &  & \tabularnewline
\cline{2-4} 
 &  & \textsf{x} & \textsf{o} & \tabularnewline
\multicolumn{5}{c}{}\tabularnewline
\end{tabular}\bigskip

Now player X must choose one of the options\bigskip

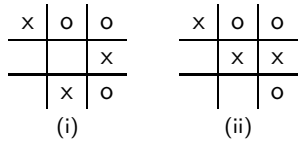
\textsf{}%
\begin{tabular}{cc|c|cc}
 & \textsf{x} & \textsf{o} & \textsf{o} & \tabularnewline
\cline{2-4} 
 &  &  & \textsf{x} & \tabularnewline
\cline{2-4} 
 &  & \textsf{x} & \textsf{o} & \tabularnewline
\multicolumn{5}{c}{\textsf{\footnotesize{}(i)}}\tabularnewline
\end{tabular}\textsf{}%
\begin{tabular}{cc|c|cc}
 & \textsf{x} & \textsf{o} & \textsf{o} & \tabularnewline
\cline{2-4} 
 &  & \textsf{x} & \textsf{x} & \tabularnewline
\cline{2-4} 
 &  &  & \textsf{o} & \tabularnewline
\multicolumn{5}{c}{\textsf{\footnotesize{}(ii)}}\tabularnewline
\end{tabular}\bigskip

For game (i), player O forces a Loop by\bigskip

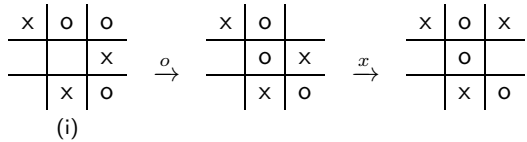
\textsf{}%
\begin{tabular}{cc|c|cc}
 & \textsf{x} & \textsf{o} & \textsf{o} & \tabularnewline
\cline{2-4} 
 &  &  & \textsf{x} & \tabularnewline
\cline{2-4} 
 &  & \textsf{x} & \textsf{o} & \tabularnewline
\multicolumn{5}{c}{\textsf{\footnotesize{}(i)}}\tabularnewline
\end{tabular}\textsf{$\xrightarrow{o}$}%
\begin{tabular}{cc|c|cc}
 & \textsf{x} & \textsf{o} &  & \tabularnewline
\cline{2-4} 
 &  & \textsf{o} & \textsf{x} & \tabularnewline
\cline{2-4} 
 &  & \textsf{x} & \textsf{o} & \tabularnewline
\multicolumn{5}{c}{}\tabularnewline
\end{tabular}\textsf{$\xrightarrow{x}$}%
\begin{tabular}{cc|c|cc}
 & \textsf{x} & \textsf{o} & \textsf{x} & \tabularnewline
\cline{2-4} 
 &  & \textsf{o} &  & \tabularnewline
\cline{2-4} 
 &  & \textsf{x} & \textsf{o} & \tabularnewline
\multicolumn{5}{c}{}\tabularnewline
\end{tabular}\bigskip

For game (ii) player O creates a Loop by\bigskip

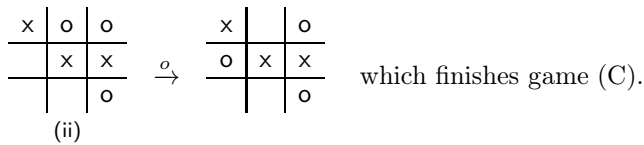
\textsf{}%
\begin{tabular}{cc|c|cc}
 & \textsf{x} & \textsf{o} & \textsf{o} & \tabularnewline
\cline{2-4} 
 &  & \textsf{x} & \textsf{x} & \tabularnewline
\cline{2-4} 
 &  &  & \textsf{o} & \tabularnewline
\multicolumn{5}{c}{\textsf{\footnotesize{}(ii)}}\tabularnewline
\end{tabular}\textsf{$\xrightarrow{o}$}%
\begin{tabular}{cc|c|cc}
 & \textsf{x} &  & \textsf{o} & \tabularnewline
\cline{2-4} 
 & \textsf{o} & \textsf{x} & \textsf{x} & \tabularnewline
\cline{2-4} 
 &  &  & \textsf{o} & \tabularnewline
\multicolumn{5}{c}{}\tabularnewline
\end{tabular} which finishes game (C).\bigskip

For game (D), player O forces a Loop by\bigskip

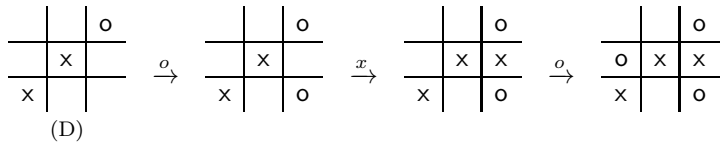
\textsf{}%
\begin{tabular}{cc|c|cc}
 &  &  & \textsf{o} & \tabularnewline
\cline{2-4} 
 &  & \textsf{x} &  & \tabularnewline
\cline{2-4} 
 & \textsf{x} &  &  & \tabularnewline
\multicolumn{5}{c}{{\footnotesize{}(D)}}\tabularnewline
\end{tabular}\textsf{$\xrightarrow{o}$}%
\begin{tabular}{cc|c|cc}
 &  &  & \textsf{o} & \tabularnewline
\cline{2-4} 
 &  & \textsf{x} &  & \tabularnewline
\cline{2-4} 
 & \textsf{x} &  & \textsf{o} & \tabularnewline
\multicolumn{5}{c}{}\tabularnewline
\end{tabular}\textsf{$\xrightarrow{x}$}%
\begin{tabular}{cc|c|cc}
 &  &  & \textsf{o} & \tabularnewline
\cline{2-4} 
 &  & \textsf{x} & \textsf{x} & \tabularnewline
\cline{2-4} 
 & \textsf{x} &  & \textsf{o} & \tabularnewline
\multicolumn{5}{c}{}\tabularnewline
\end{tabular}\textsf{$\xrightarrow{o}$}%
\begin{tabular}{cc|c|cc}
 &  &  & \textsf{o} & \tabularnewline
\cline{2-4} 
 & \textsf{o} & \textsf{x} & \textsf{x} & \tabularnewline
\cline{2-4} 
 & \textsf{x} &  & \textsf{o} & \tabularnewline
\multicolumn{5}{c}{}\tabularnewline
\end{tabular} \bigskip

which concludes the proof.
\end{proof}

\subsection{Summing up the results}
\begin{thm}
\label{thm:Player-X-cannot}Player X cannot win.\end{thm}
\begin{proof}
This follows by combining Lemmas \ref{lem:secondOcenter}, \ref{thm:OmiddleXcorner}
and \ref{lem:delicate}.\bigskip\end{proof}
\begin{cor}
Picaria is a draw.\end{cor}
\begin{proof}
This is a consequence of Theorems \ref{lem:A} and \ref{thm:Player-X-cannot}.
\end{proof}

\section{\label{sec:counting}The number of positions in the second phase}

To count the number of positions we apply Burnside's Lemma to count
equivalent classes of positions. Besides, we only count the number
of positions for the second phase, i.e., when all six pieces are placed
on the board. We consider the set $P$ the possible positions in the
game and $G$ the group acting on $P$. Denote by $\mathrm{Fix}(g)$
the set of positions in $P$ that are fixed by $g$. Burnside's Lemma
states that the number of orbits is 
\[
\#\mathrm{{orb}}=\frac{1}{\left|G\right|}\sum_{g\in G}\left|\mathrm{Fix}(g)\right|.
\]
The elements in $G$ are composed by the identity, denoted by $e$,
three rotations of the board of $90$, $180$, and $270$ degrees
denoted by $R_{90}$, $R_{180}$, and $R_{270}$. Besides, there are
four reflections, one horizontal $H$, one vertical $V$, and two
diagonal $D_{1}$ and $D_{2}$. Thus

\begin{eqnarray*}
\#\mathrm{{orb}} & = & \frac{1}{8}(\left|\mathrm{Fix}(e)\right|+\left|\mathrm{Fix}(R_{90})\right|+\left|\mathrm{Fix}(R_{180})\right|+\left|\mathrm{Fix}(R_{270})\right|\\
 &  & +\left|\mathrm{Fix}(H)\right|+\left|\mathrm{Fix}(V)\right|+\left|\mathrm{Fix}(D_{1})\right|+\left|\mathrm{Fix}(D_{2})\right|).
\end{eqnarray*}

For the identity $e$ we have $\binom{9}{3,3,3}=\frac{9!}{3!3!3!}$
positions. For $R_{90}$ and $R_{180}$ rotations with a fixed position
the board is of the respective form 

\begin{center}
\textsf{}%
\begin{tabular}{cc|c|cc}
 & b & a & d & \tabularnewline
\cline{2-4} 
 & a & e & c & \tabularnewline
\cline{2-4} 
 & b & c & d & \tabularnewline
\multicolumn{5}{c}{}\tabularnewline
\end{tabular}\textsf{ }and\textsf{ }%
\begin{tabular}{cc|c|cc}
 & d & b & c & \tabularnewline
\cline{2-4} 
 & a & e & a & \tabularnewline
\cline{2-4} 
 & c & b & d & \tabularnewline
\multicolumn{5}{c}{}\tabularnewline
\end{tabular}
\par\end{center}

Here the letters can be X, O, or empty. Notice that $\left|R_{90}\right|=\left|R_{270}\right|$.
If we put all pieces on one of these boards, there would be 4 pairs,
each pair with the same symbol, which is not possible. Thus, the rotations
do not fix any position. 

For diagonal reflections the number of fixed positions are the same
for $D_{1}$ and $D_{2}$. Consider the diagonal $D_{1}$ and horizontal
reflections with a fixed position of the respective forms

\begin{center}
\textsf{}%
\begin{tabular}{cc|c|cc}
 & b & c & f & \tabularnewline
\cline{2-4} 
 & a & e & c & \tabularnewline
\cline{2-4} 
 & d & a & b & \tabularnewline
\multicolumn{5}{c}{}\tabularnewline
\end{tabular}\textsf{ }and\textsf{ }%
\begin{tabular}{cc|c|cc}
 & a & b & c & \tabularnewline
\cline{2-4} 
 & d & e & f & \tabularnewline
\cline{2-4} 
 & a & b & c & \tabularnewline
\multicolumn{5}{c}{}\tabularnewline
\end{tabular}
\par\end{center}

These are the only fixed positions by these group actions. Here the
letters can be X, O, or empty. It follows that a,b, and c are presicely
one each of X, O, or empty, and the same is true for d,e, and f. Thus,
there are $3!3!$ such positions of each kind. Since there are four
reflections, we have in total $3!3!4$ fixed positions. 

Finally, we have 
\begin{eqnarray*}
\#\mathrm{{orb}} & = & \frac{1}{8}(\left|\mathrm{Fix}(e)\right|+\left|\mathrm{{Fix}}(R_{90})\right|+\left|\mathrm{{Fix}}(R_{180})\right|+\left|\mathrm{{Fix}}(R_{270})\right|\\
 &  & +\left|\mathrm{{Fix}}(H)\right|+\left|\mathrm{{Fix}}(V)\right|+\left|\mathrm{{Fix}}(D_{1})\right|+\left|\mathrm{{Fix}}(D_{2})\right|)\\
 & = & \frac{1}{8}\left(\frac{9!}{3!3!3!}+3!3!4\right)=228.
\end{eqnarray*}

Finally, we counted three positions that do not occur. Namely

\noindent \begin{center}
\textsf{}%
\begin{tabular}{cc|c|cc}
 &  & \textsf{o} & \textsf{x} & \tabularnewline
\cline{2-4} 
 &  & \textsf{o} & \textsf{x} & \tabularnewline
\cline{2-4} 
 &  & \textsf{o} & \textsf{x} & \tabularnewline
\multicolumn{5}{c}{}\tabularnewline
\end{tabular}\textsf{ }%
\begin{tabular}{cc|c|cc}
 & \textsf{o} &  & \textsf{x} & \tabularnewline
\cline{2-4} 
 & \textsf{o} &  & \textsf{x} & \tabularnewline
\cline{2-4} 
 & \textsf{o} &  & \textsf{x} & \tabularnewline
\multicolumn{5}{c}{}\tabularnewline
\end{tabular}\textsf{ }%
\begin{tabular}{cc|c|cc}
 &  & \textsf{x} & \textsf{o} & \tabularnewline
\cline{2-4} 
 &  & \textsf{x} & \textsf{o} & \tabularnewline
\cline{2-4} 
 &  & \textsf{x} & \textsf{o} & \tabularnewline
\multicolumn{5}{c}{}\tabularnewline
\end{tabular}\textsf{ }
\par\end{center}

Thus we get $225$ orbits. Now, the position graph for the second
phase of the game contains positions where player X or O is to move.
Thus, the position graph for this game contains $450$ positions.

\section{Final Remarks}

Consider the following natural generalizations of Picaria. We use
the same set of rules and only change the number of sides of the board
as depicted.\smallskip

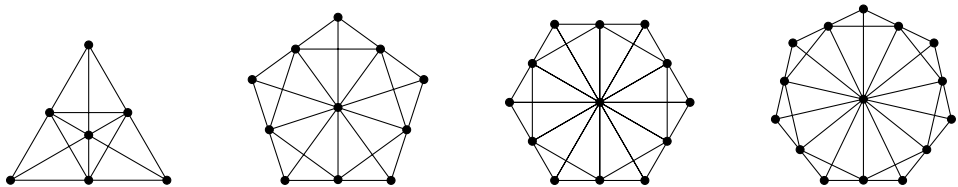
\begin{figure}[H]
\tikzstyle{every node}=[circle, draw, fill=black, inner sep=0pt, minimum width=3pt]
\begin{tikzpicture}[scale=0.3]
\draw[black] \foreach \x in {90,210,330} {         
(\x:4) node{} -- (\x+120:4)        
};
\draw[black] \foreach \x in {30,150,270} {  
(\x:2) node{} -- (\x+120:2)   
(\x:2)  -- (\x+180:4) 
};
\draw[black]{node{}
}; 
\end{tikzpicture}\hspace{9 mm}
\begin{tikzpicture}[scale=0.3]
\draw [black]\foreach \x in {18,90,...,306} {         
(\x:4) node{} -- (\x+72:4)        
};  
\draw [black]\foreach \x in {54,126,...,342} {         
(\x:3.2)  -- (\x+180:4)         
(\x:3.2) node{} -- (\x+72:3.2)    
};    
\draw[black]{node{}
};
\end{tikzpicture}\hspace{9 mm}
\begin{tikzpicture}[scale=0.3]     
\draw[black] \foreach \x in  {60,120,...,360}{        
(\x:4) node{} -- (\x+60:4)       
(\x:4)  -- (\x+180:4) 
};  
\draw[black] \foreach \x in {30,90,...,330} {  
(\x:3.45) node{} -- (\x+60:3.45)  
(\x:3.45)  -- (\x+180:3.45) 
}; 
\draw[black]{node{}
}; 
\end{tikzpicture}\hspace{9 mm}
\begin{tikzpicture}[scale=0.3] \def\a {51.4285714286}; \def\b {64.2857142857}; \def\c {128.571428572};    \draw [black]\foreach \x in {12.8571428571,\b,...,372.8571428571} {         (\x:3.59) node{} -- (\x+\a:3.59) };  \draw[black] \foreach \x in {192.8571428571,244.2857142857,...,552.8571428571} {         (\x:4) node{} -- (\x+\a:4)         (\x:4)  -- (\x+180:3.59) }; \draw[black]{node{}}; \end{tikzpicture}\vspace{12 mm}

\caption{Relatives of Picaria, with three, five, six and seven sides respectively.}
\end{figure}

\vspace{- 5 mm}Following this pattern, there are infinitely many
different board games in this family. One can show that all these
games are first player win in a few moves, except for Picaria, and
we invite the reader to find the play-proofs of this fact. We find
it compelling that the Zuni tribe played the only insteresting game
in this family for centuries.

\subsection{Open problems }

What happens if we increase the number of stones for each player,
say that game parameters, $k\ge 3$ stones each and $s\ge 3$ sides,
are given (otherwise the same rules). Is there any combination $(k,s)$,
other than $(3,4)$, for which the game is a draw (provided that the
total number of stones is less than the number of nodes)? Is it true
that the second player never wins? If we give the second player a
one stone advantage (handicap), for which combination $(k,s)$ can
he draw/win the $((k,k+1),s)$ game (that is the second player places
his last placement stone after the first slide-along-edge move by
ther first player)? In general, how many stones advantage $l > 0$
does he require to draw/win a generalized Picaria?

\end{document}